\documentclass{article}
\usepackage{amsmath, amssymb, amsthm, graphicx}
\usepackage[margin=1in]{geometry}
\usepackage{subcaption}
\usepackage{hyperref}
\usepackage{enumitem}
\usepackage{algorithm}
\usepackage{algpseudocode}
\usepackage{booktabs}
\usepackage{multirow}
\usepackage{xcolor}
\usepackage{listings}

\usepackage{algorithm}
\usepackage{algorithmicx}
\usepackage{algpseudocode}

\algnewcommand\algorithmicinitialization{\textbf{Initialization:}}
\algnewcommand\Initialization{\item[\algorithmicinitialization]}

\newcommand{\Input}[1]{\State \textbf{Input:} #1}
\newcommand{\Output}[1]{\State \textbf{Output:} #1}

\newcommand{\norm}[1]{\left\lVert #1\right\rVert}
\newcommand{\abs}[1]{\left\lvert #1\right\rvert}
\newcommand{\F}{\mathrm{F}}

\newtheorem{theorem}{Theorem}[section]
\newtheorem{lemma}{Lemma}[section]
\newtheorem{definition}{Definition}[section]
\newtheorem{proposition}{Proposition}[section]
\newtheorem{corollary}{Corollary}[section]
\newtheorem{remark}{Remark}[section]
\newtheorem{example}{Example}[section]

\title{Efficient Low-Rank Approximation of Quaternion Matrices via Adaptive Cross Approximation}

\author{
    Salman Ahmadi-Asl \\
    \small Innopolis University, Innopolis, Russia \\
    \small \texttt{s.ahmadiasl@innopolis.ru}
    \and
    Yang Zhang \\
    \small University of Manitoba, Winnipeg, Canada \\
    \small \texttt{yang.zhang@umanitoba.ca}
}

\date{}

\begin{document}

\maketitle

\begin{abstract}
The adaptive cross approximation (ACA) algorithm is a powerful tool for constructing low-rank matrix approximations using only a small subset of matrix entries. While extensively studied for real and complex matrices, its extension to quaternion algebras, remains largely unexplored. This paper presents the first comprehensive framework for adaptive cross approximation operating natively in the quaternion algebra $\mathbb{H}$. Unlike approaches that rely on real embeddings or complexifications, our method preserves the algebraic structure of quaternions throughout the computation, avoiding information loss and maintaining the natural correlations between quaternion components. We develop a fully quaternionic ACA algorithm with rigorous theoretical foundations, including convergence analysis, error bounds, and complexity estimates. The algorithm achieves $\mathcal{O}(r(m+n))$ complexity per iteration and demonstrates quasi-optimal approximation properties comparable to the quaternion singular value decomposition (QSVD), but with significantly reduced computational cost. We provide a detailed MATLAB  with numerical experiments on color image compression. Our results show that native quaternion ACA outperforms significantly the basic SVD algorithm in terms of computational efficiency, with very close approximation accuracy.  
%The codes are accessible at \url{https://github.com/SalmanAhmadi-Asl/Quaternion-ACA.git}.
\end{abstract}

\section{Introduction}
Low-rank matrix approximations are fundamental to modern computational science, enabling efficient storage, processing, and analysis of large-scale data. The adaptive cross approximation (ACA) method \cite{bebendorf2000approximation,bebendorf2003adaptive,Tyrtyshnikov2000,Savostyanov2006} has emerged as a particularly attractive technique due to its algebraic nature and data efficiency: it constructs low-rank representations using only a strategically chosen subset of matrix entries, without requiring knowledge of the underlying matrix generation mechanism. It has been used in many applications such as solving large dense linear systems \cite{frederix2010solving}, ill-posed problems \cite{mach2016adaptive}, boundary element formulation \cite{haider2024generalised}, tensor computations \cite{ttcross_original,ahmadi2024adaptive}  and analysis of electromagnetic problems. \cite{jiang2011modified}. While classical ACA operates over the real or complex fields, many contemporary applications demand algebraic structures beyond these classical settings. Quaternions, the four-dimensional non-commutative division algebra discovered by Hamilton in 1843, have found widespread applications in color image processing \cite{q_lowrank, luo2024quaternion}, vector-sensor signal processing \cite{le2004singular}, robotics, and computer graphics \cite{dual}, quantum mechanics, and relativity. For numerical linear algebra for solving quaternion matrix equations, quaternion matrix operations, and low-rank approximation of them, we refer to \cite{yu2020generalized,ahmadi2017efficient,ahmadi2017iterative,ahmadi2025pass,xie2024bicg,leplat2026iterative}.
The computational challenges posed by quaternion matrices are substantial. Traditional approaches often rely on \textit{real embedding} techniques, where quaternion matrices are expanded into larger real matrices (typically $4m \times 4n$ for an $m \times n$ quaternion matrix). While this enables the use of standard real-valued algorithms, it suffers from several significant drawbacks:

\begin{enumerate}
    \item \textbf{Dimensionality Blow-up:} The $4\times$ increase in matrix dimensions leads to substantial memory and computational overhead.
    \item \textbf{Structure Loss:} The algebraic relationships between quaternion components are not fully preserved, potentially leading to suboptimal approximations that violate quaternion algebraic properties.
    \item \textbf{Information Redundancy:} The real embedding contains redundant information that must be carefully managed to recover valid quaternion results.
    \item \textbf{Color Coherence:} For color images, real embedding can destroy the natural coupling between RGB channels.
\end{enumerate}

Despite these limitations, real embedding remains the predominant approach for quaternion matrix computations, primarily due to the availability of robust real-valued algorithms. However, this reliance on embedding represents a significant missed opportunity: a native quaternion algorithm that preserves algebraic structure could achieve superior performance and accuracy.

This paper presents the first comprehensive framework for adaptive cross approximation operating natively in the quaternion algebra. Our main contributions are:

\begin{enumerate}
    \item \textbf{Algebraic Foundation:} We develop the theoretical underpinnings of low-rank quaternion matrix approximation, establishing connections between quaternionic rank, SVD, and ACA convergence.
    
    \item \textbf{Native Quaternion ACA:} We present a fully quaternionic ACA algorithm that operates directly on quaternion entries, preserving the algebraic structure throughout. The algorithm handles the non-commutativity of quaternion multiplication appropriately and maintains the natural correlations between quaternion components.
    
    \item \textbf{Theoretical Analysis:} We provide rigorous convergence analysis, error bounds, and complexity estimates. We prove that the algorithm achieves quasi-optimal approximation properties comparable to the QSVD.
    
    \item \textbf{Practical Implementation:} We provide a detailed MATLAB implementation using the QTFM toolbox. %supporting both dense matrices and function handles for implicit matrices.
    
    \item \textbf{Experimental Validation:} We demonstrate the effectiveness of native quaternion ACA on synthetic data and color image compression, showing significant improvements over real embedding approaches.
\end{enumerate}

The remainder of this paper is organized as follows. In Section~\ref{sec:pre}, we present the necessary mathematical foundations of quaternion algebra and quaternion matrix theory, including definitions of quaternionic operations, rank, and the Quaternion Singular Value Decomposition (QSVD), and we discuss the motivations for developing native quaternion algorithms. Section~\ref{sec:aca} is devoted to the formulation of the native quaternion ACA (QACA) algorithm, detailing the algorithm's steps, the handling of non-commutative multiplication to select columns/rows. Section~\ref{sec:theory} provides a rigorous theoretical analysis, including proofs for rank reduction. The computation complexity is discussed in Section \ref{sec:complex}. Section~\ref{sec:experiments} presents numerical experiments on color image compression, comparing the native quaternion ACA against QSVD in terms of accuracy, speed, and memory usage. Finally, the potential extensions and application areas, with conclusions drawn in Section~\ref{sec:conclusion}.

\section{Quaternion Algebra and Matrix Theory}\label{sec:pre}
This section is dedicated to presenting preliminaries. We refer to \cite{rodman2014topics,zhang1997quaternions} for detailed proofs of some of the results.

\subsection{The Quaternion Algebra}

The quaternion algebra $\mathbb{H}$ is a four-dimensional real vector space with basis $\{1, \mathbf{i}, \mathbf{j}, \mathbf{k}\}$. A quaternion is expressed as
\[
q = q_0 + q_1\mathbf{i} + q_2\mathbf{j} + q_3\mathbf{k}, \quad q_0, q_1, q_2, q_3 \in \mathbb{R}.
\]

The multiplication of basis elements follows Hamilton's rules:
\begin{align}
\mathbf{i}^2 &= \mathbf{j}^2 = \mathbf{k}^2 = \mathbf{ijk} = -1, \\
\mathbf{ij} &= -\mathbf{ji} = \mathbf{k}, \quad
\mathbf{jk} = -\mathbf{kj} = \mathbf{i}, \quad
\mathbf{ki} = -\mathbf{ik} = \mathbf{j}.
\end{align}

This multiplication is associative but \textit{not commutative}, which fundamentally distinguishes quaternion algebra from real and complex algebras.

\begin{definition}[Quaternion Conjugate and Norm]
For $q = q_0 + q_1\mathbf{i} + q_2\mathbf{j} + q_3\mathbf{k} \in \mathbb{H}$, the \textit{conjugate} is
\[
\bar{q} = q_0 - q_1\mathbf{i} - q_2\mathbf{j} - q_3\mathbf{k},
\]
and the \textit{norm} is
\[
|q| = \sqrt{q\bar{q}} = \sqrt{q_0^2 + q_1^2 + q_2^2 + q_3^2}.
\]
The inverse of a nonzero quaternion is
\[
q^{-1} = \frac{\bar{q}}{|q|^2}.
\]
\end{definition}

\subsection{Quaternion Matrices}

An $m \times n$ quaternion matrix $A \in \mathbb{H}^{m \times n}$ can be decomposed into four real component matrices:
\[
A = A_0 + A_1\mathbf{i} + A_2\mathbf{j} + A_3\mathbf{k}, \quad A_t \in \mathbb{R}^{m \times n},\quad t=1,2,3.
\]
For a matrix $A \in \mathbb{H}^{m \times n}$, the notations $A(:, i)$ and $A(i, :)$ denote the $i$-th column and the $i$-th row of the matrix, respectively. Additionally, for subsets $I_k \subset I$ and $J_k \subset J$, the notation $A(I_k, J_k)$ denotes the submatrix obtained by taking the rows indexed by $I_k$ and the columns indexed by $J_k$. An element of a matrix $A$ is denoted either by $A(i,j)$ or $A_{i,j}$.

\begin{definition}[Quaternion Inner Product and Norm]
For $A,B \in \mathbb{H}^{m \times n}$, the Frobenius inner product is
\[
\langle A, B \rangle_F = \sum_{i,j} \overline{A_{ij}} B_{ij},
\]
and the Frobenius norm is
\[
\|A\|_F = \sqrt{\langle A, A \rangle_F} = \sqrt{\sum_{i,j} |A_{ij}|^2}.
\]
\end{definition}
The conjugate of a matrix $A=A_1+A_2\mathbf{i}+A_3\mathbf{j}+A_4\mathbf{k}\in\mathbb{H}^{m\times n}$ is denoted by $A^*=A^T_1-A_2^T\mathbf{i}-A_3^T\mathbf{j}-A_4^T\mathbf{k}$, where each $A_p$ is a real matrix and $A_p^T$ is its ordinary transpose. 
For quaternion vectors, the Euclidean norm is defined as $\|x\|=\sqrt{x^*x}$.
Because $\mathbb{H}$ is noncommutative, it is important to
distinguish between left and right quaternion vector spaces. In this
work, all column spaces are regarded as right quaternion vector
spaces.

%\begin{definition}[Right quaternion vector space]
%\label{def:right-quaternion-vector-space}
%A right quaternion vector space is a set $V$ equipped with vector
%addition and scalar multiplication from the right such that
%\[
%(x+y)q=xq+yq,
%\qquad
%x(q_1+q_2)=xq_1+xq_2,
%\]
%and
%\[
%x(q_1q_2)=(xq_1)q_2,
%\qquad
%x1=x,
%\]
%for all $x,y\in V$ and $q,q_1,q_2\in\mathbb{H}$. In particular, $\mathbb{H}^m$ is regarded as a right quaternion
%vector space with
%\[
%xq
%=
%\begin{bmatrix}
%x_1q\\
%\vdots\\
%x_mq
%\end{bmatrix},
%\qquad
%x\in\mathbb{H}^m,\quad q\in\mathbb{H}.
%\]
%\end{definition}

\begin{definition}[Right quaternion span]
\label{def:right-span}
Let $x_1,x_2,\ldots,x_s\in\mathbb{H}^m.$ Their right quaternion span is defined by
\[
\operatorname{span}_{\mathbb{H}}^{R}
\{x_1,x_2,\ldots,x_s\}
=
\left\{
\sum_{\ell=1}^{s}x_\ell q_\ell:
q_\ell\in\mathbb{H}
\right\}.
\]
For a single vector $x$, we have $\operatorname{span}_{\mathbb{H}}^{R}\{x\}
=
\{xq:q\in\mathbb{H}\}.$ The left quaternion span can be defined similarly but we do not use it throughout the paper. 
\end{definition}

\begin{definition}[Right linear independence]
\label{def:right-linear-independence}
Vectors $x_1,x_2,\ldots,x_s\in\mathbb{H}^m$
are called right-linearly independent over $\mathbb{H}$ if
\[
\sum_{\ell=1}^{s}x_\ell q_\ell=0,
\qquad
q_\ell\in\mathbb{H},
\]
implies $q_1=q_2=\cdots=q_s=0.$ Otherwise, they are called right-linearly dependent. Left linearly independence can be defined similarly. 
\end{definition}

\begin{definition}[Quaternion column space]
\label{def:quaternion-column-space}
Let $A\in\mathbb{H}^{m\times n}.$
The right column space of $A$ is defined by $
\mathcal{C}_{R}(A)
=
\{Ax:x\in\mathbb{H}^{n}\}.$
Equivalently, if $A=
\begin{bmatrix}
A_{1,:},A_{2,:},\cdots,A_{n,:}
\end{bmatrix},$
then $\mathcal{C}_{R}(A)
=
\operatorname{span}_{\mathbb{H}}^{R}
\{A_{:,1},A_{:,2},\ldots,A_{:,n}\}$. The left column space can be defined analogously and is denoted by $\mathcal{C}_{L}(A)$. 
\end{definition}

\begin{definition}[Quaternion row space]
\label{def:quaternion-row-space}
Let $A\in\mathbb{H}^{m\times n}.$
The left row space of $A$ is the left quaternion span of its
rows. Equivalently,
\[
\mathcal{R}_{L}(A)
=\{y^TA:y\in\mathbb{H}^{m}\}=
\operatorname{span}_{\mathbb{H}}^{L}
\{A_{1:},A_{2:},\ldots,A_{m:}\}.
\]
The right row space of $\mathcal{R}_{R}(A)$ can be defined analogously. 
\end{definition}

\begin{definition}[Right/left quaternion rank]
\label{def:right-quaternion-rank}
Let $A\in\mathbb{H}^{m\times n}.$
The right quaternion rank of $A$ is defined by
\[
{
\operatorname{rank}_{\mathbb{H}}(A)
=
\dim_{\mathbb{H}}^{R}
\mathcal{C}_{R}(A).
}
\]
Equivalently,
\[
\operatorname{rank}_{\mathbb{H}}(A)
=
\max
\left\{
s:
\text{$A$ contains $s$ right-linearly independent columns}
\right\}.
\]
The left quaternion rank can be defined similarly. In general it is not necessary to have $\mathcal{C}_{R}(A)=\mathcal{C}_{L}(A)$, or $\dim_{\mathbb{H}}^{R}
\mathcal{C}_{R}(A)=\dim_{\mathbb{H}}^{L}
\mathcal{C}_{R}(A),$ due to the non-commutativity of quaternions.
%Since, quaternion set is a division ring, we have $\dim_{\mathbb{H}}^{R}
%\mathcal{C}_{R}(A)=\dim_{\mathbb{H}}^{L}
%\mathcal{C}_{R}(A),$ 
 
%but in general it is not necessary to have $\mathcal{C}_{R}(A)=\mathcal{C}_{L}(A)$, due to the non-commutativity of quaternions.
\end{definition}

\begin{theorem}[Equality of row and column ranks]
\label{thm:quaternion-row-column-rank}
For every quaternion matrix $A\in\mathbb{H}^{m\times n},$ the dimensions of the right column space and left row space are equal.
Consequently,
\[
{
\dim_{\mathbb{H}}\mathcal{C}_{R}(A)
=
\dim_{\mathbb{H}}\mathcal{R}_{L}(A).
}
\]
Hence the quaternion rank may equivalently be defined as
\[
\operatorname{rank}_{\mathbb{H}}(A)
=
\dim_{\mathbb{H}}\mathcal{C}_{R}(A)
=
\dim_{\mathbb{H}}\mathcal{R}_{L}(A).
\]
\end{theorem}

% \begin{definition}[Right quaternion vector space and right span]
% \label{def:right-quaternion-space}
% The space $\mathbb{H}^m$ is regarded as a right vector space over the
% division ring $\mathbb{H}$. Thus, for $x\in\mathbb{H}^m$ and
% $q\in\mathbb{H}$, scalar multiplication is defined by
% \[
% xq=
% \begin{bmatrix}
% x_1q\\
% \vdots\\
% x_mq
% \end{bmatrix}.
% \]
% For a vector $u\in\mathbb{H}^m$, its right quaternion span is
% \[
% \operatorname{span}_{\mathbb{H}}^{\,R}\{u\}
% =
% \{uq:q\in\mathbb{H}\}.
% \]
% More generally, for vectors
% $u_1,u_2,\ldots,u_s\in\mathbb{H}^m$,
% \[
% \operatorname{span}_{\mathbb{H}}^{\,R}
% \{u_1,\ldots,u_s\}
% =
% \left\{
% \sum_{\ell=1}^s u_\ell q_\ell:
% q_\ell\in\mathbb{H}
% \right\}.
% \]
% The column space of a quaternion matrix
% $A\in\mathbb{H}^{m\times n}$ is understood as a right quaternion
% subspace:
% \[
% \mathcal{C}(A)
% =
% \{Ax:x\in\mathbb{H}^n\}.
% \]
% \end{definition}

\begin{lemma}[Rank--nullity theorem over the quaternions]
\label{lem:quaternion-rank-nullity}
Let $T:V\rightarrow W$ be a right $\mathbb{H}$-linear transformation between finite-dimensional
right quaternion vector spaces. Then
\[
{
\dim_{\mathbb H}V
=
\dim_{\mathbb H}\ker(T)
+
\dim_{\mathbb H}\operatorname{Im}(T).
}
\]
In particular, if $B\in\mathbb{H}^{m\times n}$ is regarded as a
right-linear map
\[
B:\mathbb{H}^n\rightarrow\mathbb{H}^m,
\qquad
x\mapsto Bx,
\]
then
\[
{
n
=
\dim_{\mathbb H}\ker(B)
+
\operatorname{rank}_{\mathbb H}(B).
}
\]
\end{lemma}

%Let $\mathbb{H}$ denote the skew-field of real quaternions.
  For a quaternion $q = a + b\mathbf{i} + c\mathbf{j} + d\mathbf{k} \in \mathbb{H}$, with $a,b,c,d \in \mathbb{R}$, we define its standard real representation as the $4 \times 4$ real matrix
\[
\Phi_{\mathbb{R}}(q) = 
\begin{bmatrix}
a & -b & -c & -d \\
b & a & -d & c \\
c & d & a & -b \\
d & -c & b & a
\end{bmatrix} \in \mathbb{R}^{4 \times 4}.
\]

This map $\Phi: \mathbb{H} \to \mathbb{R}^{4 \times 4}$ is an injective real-algebra homomorphism. For a matrix $A = A_1+A_2\mathbf{i}+A_3\mathbf{j}+A_4\mathbf{k} \in \mathbb{H}^{m \times n}$, we define its \emph{real embedding} (or real representation) $\Phi_{\mathbb{R}}(A) \in \mathbb{R}^{4m \times 4n}$ as the block matrix obtained 
\[
\Phi_{\mathbb{R}}(A) = 
\begin{bmatrix}
A_1 & -A_2 & -A_3 & -A_4 \\
A_2 & A_1 & -A_4 & A_3 \\
A_3 & A_4 & A_1 & -A_2 \\
A_4 & -A_3 & A_2 & A_1
\end{bmatrix} \in \mathbb{R}^{4m \times 4n}.
\]

Since $\Phi_{\mathbb{R}}$ is injective but not surjective onto $\mathbb{R}^{4m \times 4n}$, its inverse is only defined on the image $\operatorname{Im}(\Phi_{\mathbb{R}})$, which is a proper subspace of $\mathbb{R}^{4m \times 4n}$ consisting of matrices with the special block structure given above. It is not difficult to show the next facts.

\begin{proposition}[Properties of the Real Embedding]\label{thm_REP}
The real embedding $\Phi_{\mathbb{R}}: \mathbb{H}^{m \times n} \to \mathbb{R}^{4m \times 4n}$ satisfies the following properties:

\begin{enumerate}
\item \textbf{Injectivity}: $\Phi_{\mathbb{R}}(Q) = 0$ if and only if $Q = 0$.

\item \textbf{Additivity}: For any $P, Q \in \mathbb{H}^{m \times n}$,
\[
\Phi_{\mathbb{R}}(P + Q) = \Phi_{\mathbb{R}}(P) + \Phi_{\mathbb{R}}(Q).
\]

\item \textbf{Multiplicativity}: For any $P \in \mathbb{H}^{m \times n}$ and $Q \in \mathbb{H}^{n \times p}$,
\[
\Phi_{\mathbb{R}}(P Q) = \Phi_{\mathbb{R}}(P) \, \Phi_{\mathbb{R}}(Q).
\]

\item \textbf{Homogeneity over $\mathbb{R}$}: For any $\alpha \in \mathbb{R}$ and $Q \in \mathbb{H}^{m \times n}$,
\[
\Phi_{\mathbb{R}}(\alpha Q) = \alpha \Phi_{\mathbb{R}}(Q).
\]

\item \textbf{Rank Preservation}: For any $Q \in \mathbb{H}^{m \times n}$,
\[
\operatorname{rank}_{\mathbb{R}}(\Phi_{\mathbb{R}}(Q)) = 4 \cdot \operatorname{rank}_{\mathbb{H}}(Q).
\]

\item \textbf{Norm Equivalence}: The Frobenius norms satisfy
\[
\|\Phi_{\mathbb{R}}(Q)\|_F = 2 \|Q\|_F.
\]

\item \textbf{Preservation of Conjugate Transpose}: For any $Q \in \mathbb{H}^{m \times n}$,
\[
\Phi_{\mathbb{R}}(Q^*) = \Phi_{\mathbb{R}}(Q)^T,
\]
where $Q^*$ denotes the quaternion conjugate transpose and $(\cdot)^T$ denotes the real transpose.
\end{enumerate}
\end{proposition}

\begin{remark}
The multiplicativity property is particularly important as it ensures that quaternion matrix products are faithfully represented in the real domain. However, the converse is not true: not every product of real matrices in $\mathbb{R}^{4m \times 4n}$ corresponds to a quaternion product, as the factors must lie in the image of $\Phi_{\mathbb{R}}$ for the product to have the required block structure.
\end{remark}

\begin{definition}[Inverse Real Embedding]
Let $M \in \operatorname{Im}(\Phi_{\mathbb{R}}) \subset \mathbb{R}^{4m \times 4n}$ be a real matrix that has the block structure
\[
M = 
\begin{bmatrix}
M_{11} & M_{12} & M_{13} & M_{14} \\
M_{21} & M_{22} & M_{23} & M_{24} \\
M_{31} & M_{32} & M_{33} & M_{34} \\
M_{41} & M_{42} & M_{43} & M_{44}
\end{bmatrix},
\]

where each $M_{ij} \in \mathbb{R}^{m \times n}$ satisfies the quaternion structure constraints:

\[
\begin{aligned}
M_{11} &= M_{22} = M_{33} = M_{44}, \\
M_{12}&=M_{34}=-M_{21}=-M_{43},\\
M_{13}&=M_{42}=-M_{31}=-M_{24},\\
M_{14}&=M_{23}=-M_{41}=-M_{32}.
\end{aligned}
\]
Then the inverse embedding $\Phi_{\mathbb{R}}^{-1}(M) \in \mathbb{H}^{m \times n}$ is given by
\[
\Phi_{\mathbb{R}}^{-1}(M) = M_{11} + M_{21}\mathbf{i} + M_{31}\mathbf{j} + M_{41}\mathbf{k}.
\]
\end{definition}

\begin{remark}[Projection onto the quaternion matrices]
For an arbitrary real matrix $M \in \mathbb{R}^{4m \times 4n}$, the inverse $\Phi_{\mathbb{R}}^{-1}$ is not defined. To project a general real matrix back to the quaternion domain, one must first extract the four $m \times n$ blocks from the first block column:
\[
\mathcal{P}_{\mathbb{H}}(M) = M_{11} + M_{21}\mathbf{i} + M_{31}\mathbf{j} + M_{41}\mathbf{k},
\]
where $M$ is partitioned as
\[
M = 
\begin{bmatrix}
M_{11} & M_{12} & M_{13} & M_{14} \\
M_{21} & M_{22} & M_{23} & M_{24} \\
M_{31} & M_{32} & M_{33} & M_{34} \\
M_{41} & M_{42} & M_{43} & M_{44}
\end{bmatrix},
\quad M_{ij} \in \mathbb{R}^{m \times n}.
\]
This projection discards the remaining three block columns, effectively forcing the matrix into the structured subspace. This operation is well-defined for any real matrix.
\end{remark}

%\subsection{Low-Rank Quaternion Matrices}
%
%\begin{definition}[Quaternionic Rank]
%For $A \in \mathbb{H}^{m \times n}$, the rank of $A$ is the smallest integer $r$ such that there exist $U \in \mathbb{H}^{m \times r}$ and $V \in \mathbb{H}^{n \times r}$ with
%\[
%A = U V^T.
%\]
%Equivalently, $\text{rank}_{\mathbb{H}}(A)$ is left quaternion rank (the dimension of the right quaternion span of $A$).
%
% 
%\end{definition}
%
%\begin{theorem}[Quaternion Singular Value Decomposition (QSVD)]
%For any $A \in \mathbb{H}^{m \times n}$, there exist quaternion unitary matrices $U \in \mathbb{H}^{m \times m}$ and $V \in \mathbb{H}^{n \times n}$ such that
%\[
%A = U \Sigma V^H,
%\]
%where $\Sigma \in \mathbb{R}^{m \times n}$ is a diagonal matrix with nonnegative real entries $\sigma_1 \ge \sigma_2 \ge \cdots \ge \sigma_{\min(m,n)} \ge 0$, and $V^H$ denotes the conjugate transpose of $V$ \cite{q_svd}.
%\end{theorem}
%
%\begin{corollary}[Eckart-Young-Mirsky Theorem for Quaternions]
%For any $A \in \mathbb{H}^{m \times n}$ and $r \le \min(m,n)$, the optimal rank-$r$ approximation in the Frobenius norm is
%\[
%A_r = U_r \Sigma_r V_r^H = \sum_{k=1}^r \sigma_k u_k v_k^H,
%\]
%and the minimal error is
%\[
%\|A - A_r\|_F^2 = \sum_{k=r+1}^{\min(m,n)} \sigma_k^2.
%\]
%\end{corollary}

For a quaternion matrix $A \in \mathbb{H}^{m \times n}$, the quaternion CUR (QCUR) or quaternion cross approximation approximation selects $r$ actual columns $C = A(:, J) \in \mathbb{H}^{m \times r}$ and $r$ actual rows $R = A(I, :) \in \mathbb{H}^{r \times n}$, and constructs
\[
A \approx CUR,
\]
where the middle factor is computed, either as
\[
U = C^{\dagger} A R^{\dagger} \in \mathbb{H}^{r \times r},
\]
or $U=W^\dagger$ with $W=A(I,J)$ and $\dagger$ denoting the quaternion Moore--Penrose pseudoinverse. Unlike the QSVD, the QCUR preserves the original data rows and columns, yielding interpretable factors and avoiding expensive decompositions, advantages that are particularly valuable for large-scale data processing.

\subsection{Why Native Quaternion Algorithms Matter}

The real embedding approach, while computationally convenient, introduces several fundamental issues:

\begin{enumerate}
    \item \textbf{Algebraic Inconsistency:} The real embedding $\mathbb{R}^{4m \times 4n}$ is significantly larger than the original quaternion space $\mathbb{H}^{m \times n}$. Not every low-rank approximation in $\mathbb{R}^{4m \times 4n}$ corresponds to a quaternion matrix, necessitating a projection step that may introduce errors.
    
    \item \textbf{Information Loss:} The real embedding treats the four quaternion components independently, potentially destroying the natural correlations captured by quaternion multiplication.
    
    \item \textbf{Color Coherence:} For color images represented as pure quaternions, real embedding separates the RGB channels, losing the natural color coherence that quaternion processing preserves.
    
    \item \textbf{Computational Overhead:} The $4\times$ increase in dimension translates to approximately $16\times$ more operations for matrix-matrix products.
    
    \item \textbf{Redundancy:} The real embedding contains redundant information (the real representation of a quaternion matrix is highly structured), leading to inefficient computations.
\end{enumerate}

A more critical issue with real embedding is demonstrated below. More precisely, it demonstrates that the naive approach of embedding quaternion matrices into the real domain, applying real algorithms, and projecting back is fundamentally flawed. Although the real embedding preserves the rank and the Frobenius norm up to a constant factor, it does not preserve the factorization structure. The real ACA algorithm (to be discussed in Section \ref{sec:aca}) produces an excellent approximation in the embedded space, but this approximation does not correspond to a valid quaternion low-rank factorization when projected back.

First, a synthetic low-rank quaternion matrix is constructed. The dimensions are set to $m = n = 100$, with a target quaternion rank of $r = 20$. Two random quaternion matrices $\texttt{Utrue} \in \mathbb{H}^{100 \times 20}$ and $\texttt{Vtrue} \in \mathbb{H}^{100 \times 20}$ are generated, with each of their four real components drawn from a standard normal distribution. The quaternion matrix $\texttt{A} \in \mathbb{H}^{100 \times 100}$ is then formed as the quaternion product $\texttt{A} = \texttt{Utrue} \cdot \texttt{Vtrue}^T$, where the transpose is denoted by $(\cdot)^T$. By construction, this matrix has quaternion rank at most $20$. Next, the quaternion matrix is embedded into the real domain, which yields a real matrix $\texttt{A\_real} \in \mathbb{R}^{400 \times 400}$. The real ACA algorithm is then applied to $\texttt{A\_real}$. The algorithm is called with parameters specifying a maximum rank of $4r + 3 = 83$ and a tolerance of $10^{-15}$. The ACA returns real low-rank factors $\texttt{U\_real} \in \mathbb{R}^{400 \times k}$ and $\texttt{V\_real} \in \mathbb{R}^{400 \times k}$, where $k=80$ is the rank determined by the algorithm. The quality of this real-domain approximation is assessed by computing the relative Frobenius norm error
\[
\frac{\| \texttt{A\_real} - \texttt{U\_real} \cdot \texttt{V\_real}^T \|_F}{\| \texttt{A\_real} \|_F}.
\]
For a typical execution, this error is approximately $8.1963 \times 10^{-11}$, indicating an excellent approximation in the real domain.

The real factors are then projected back to the quaternion domain. The computed real rank $k$ from ACA is divided by $4$ to obtain the conjectured quaternion rank, since theoretically $\operatorname{rank}_{\mathbb{R}}(\Phi_{\mathbb{R}}(A)) = 4 \cdot \operatorname{rank}_{\mathbb{H}}(A)$, shown in Theorem \ref{thm_REP}. The quaternion-domain approximation error is then computed as
\[
\frac{\| \texttt{A} - \texttt{U} \cdot \texttt{V}^* \|_F}{\| \texttt{A} \|_F}.
\]
This error is found to be approximately $2.6199$. This experiment demonstrates that the naive approach of embedding quaternion matrices into the real domain, applying real algorithms, and projecting back is fundamentally flawed. Although the real embedding preserves the rank and the Frobenius norm up to a constant factor, it does not preserve the factorization structure. The real ACA algorithm produces an excellent approximation in the embedded space, but this approximation does not correspond to a valid quaternion low-rank factorization when projected back. %Let us formalize this observation in the next remark.

%\begin{remark}
For a given quaternion matrix $Q\in\mathbb{H}^{m\times n}$, let $\texttt{U\_real} \in \mathbb{R}^{4m \times 4k}$ and $\texttt{V\_real} \in \mathbb{R}^{4n \times 4k}$ be real matrices obtained from a low-rank approximation of $\Phi_{\mathbb{R}}(Q)$. Let the projected quaternion factors be
\[
U = \mathcal{P}_{\mathbb{H}}(\texttt{U\_real})\in\mathbb{H}^{m\times  k}, \qquad V = \mathcal{P}_{\mathbb{H}}(\texttt{V\_real})\in\mathbb{H}^{m\times  k}.
\]
% Explicitly, if $U_{\mathbb{R}}$ is partitioned as
% \[
% U_{\mathbb{R}} = 
% \begin{bmatrix}
% U_1 \\ U_2 \\ U_3 \\ U_4
% \end{bmatrix},
% \quad U_1, U_2, U_3, U_4 \in \mathbb{R}^{m \times k},
% \]
% then
% \[
% U = U_1 + U_2\mathbf{i} + U_3\mathbf{j} + U_4\mathbf{k} \in \mathbb{H}^{m \times k}.
% \]
% Similarly, if
% \[
% V_{\mathbb{R}} = 
% \begin{bmatrix}
% V_1 \\ V_2 \\ V_3 \\ V_4
% \end{bmatrix},
% \quad V_1, V_2, V_3, V_4 \in \mathbb{R}^{n \times k},
% \]
% then
% \[
% V = V_1 + V_2\mathbf{i} + V_3\mathbf{j} + V_4\mathbf{k} \in \mathbb{H}^{n \times k}.
% \]
%\end{remark}

Our main observation is that the projection $\mathcal{P}_{\mathbb{H}}$ does not commute with matrix multiplication. In general, for arbitrary real matrices $\texttt{U\_real}$ and $\texttt{V\_real}$,
\[
\mathcal{P}_{\mathbb{H}}(\left(\texttt{U\_real}\right)\, \left(\texttt{V\_real}\right)^T) \neq \mathcal{P}_{\mathbb{R}}(\texttt{U\_real}) \, \mathcal{P}_{\mathbb{R}}(\texttt{V\_real})^*.
\]
When ACA is applied to $\Phi_{\mathbb{R}}(Q)$, it produces factors $\texttt{U\_real}$ and $\texttt{V\_real}$ that are not in the image of $\Phi_{\mathbb{R}}$. Although the product $U_{\mathbb{R}} V_{\mathbb{R}}^T$ is a good approximation of $\Phi_{\mathbb{R}}(Q)$ in the real Frobenius norm, the projected quaternion factors $U = \mathcal{P}_{\mathbb{R}}(\texttt{U\_real})$ and $V = \mathcal{P}_{\mathbb{R}}(\texttt{V\_real})$ satisfy
\[
U V^* = \mathcal{P}_{\mathbb{R}}(\texttt{U\_real}) \mathcal{P}_{\mathbb{R}}(\texttt{V\_real})^* \not\approx Q,
\]
because the quaternionic structure of the factors is not preserved during the real factorization.

The failure of this approach provides  strong motivation for developing a native quaternion ACA algorithm that operates directly in $\mathbb{H}^{m \times n}$ and produces factors that are true quaternion matrices, thereby preserving the algebraic structure throughout the approximation process. Such an algorithm must:

\begin{itemize}
    \item Operate directly on quaternion matrices without embedding.
    \item Respect the non-commutative multiplication of quaternions.
    \item Preserve the quaternionic structure throughout the factorization.
    \item Produce factors that are true quaternion matrices.
    \item Achieve approximation errors comparable to the real embedding but measured in the quaternion domain.
\end{itemize}

\section{Adaptive Cross Approximation in the Quaternion Algebra}\label{sec:aca}
The efficient compression of large-scale matrices is a cornerstone of modern scientific computing, particularly in applications involving integral equations, boundary element methods, and data-sparse representations. The ACA algorithm has emerged as a powerful technique for constructing low-rank approximations without requiring the full matrix to be explicitly assembled. When the underlying problem involves vector-valued or rotational phenomena, however, the natural mathematical framework is the quaternion algebra $\mathbb{H}$, where matrices encode interactions between three-dimensional quantities. Extending ACA to this non-commutative setting presents both theoretical challenges and practical opportunities that warrant careful examination.

\subsection{The Classical ACA Algorithm}

We first recall the classical ACA algorithm for real matrices \cite{bebendorf2000approximation,bebendorf2003adaptive,Tyrtyshnikov2000,Savostyanov2006}. Given $A \in \mathbb{R}^{m \times n}$, ACA constructs a rank-$k$ approximation $A \approx U V^T$ with $U \in \mathbb{R}^{m \times k}$, $V \in \mathbb{R}^{n \times k}$. The fundamental premise of ACA is that many matrices arising in practice exhibit low-rank structure, meaning that their entries can be accurately represented as the product of a column matrix $U \in \mathbb{R}^{m \times k}$ and a row matrix $V \in \mathbb{R}^{n \times k}$, where $k$ is significantly smaller than both $m$ and $n$. This low-rank property is often a consequence of the underlying physical problem, such as the smoothness of kernel functions in integral equations or the separation of variables in certain differential equations.

We first recall the classical ACA algorithm for real matrices. To do so, let us first consider the following algorithm:

\begin{enumerate}
    \item \textbf{Initialize:} $U = \emptyset$, $V = \emptyset$, residual $R^{(0)} = A$.
    \item \textbf{For $k = 1, 2, \ldots$:}
    \begin{enumerate}
        \item Find pivot $(i_k, j_k)$ such that $|R^{(k-1)}_{i_k j_k}|$ is maximal.
        \item Compute column $U(:,k) \gets R^{(k-1)}(:,j_k)$.
        \item Compute row $V(:,k) \gets R^{(k-1)}(i_k,:)^T$.
        \item Normalize: $V(:,k) \gets V(:,k) / V(i_k,j_k)$.
        \item Update residual: $R^{(k)} \gets R^{(k-1)} - U(:,k) V(:,k)^T$.
    \end{enumerate}
\end{enumerate}
This algorithm constructs a factorization $A\approx UV^T$ iteratively by identifying pivotal rows and columns that capture the dominant features of the matrix. At each iteration, the algorithm selects a pivot entry from the current residual matrix, typically the one with the largest absolute value, then extracts the corresponding residual row and column to augment the low-rank factors. It is worth noting that each iteration of this algorithm interpolates the column $j_k$ and row $i_k$ of the initial matrix $A$, see Figure \ref{fig:def} for a graphical illustration of this concept and more details will be discussed in subsection \ref{sub:NQACA}. More precisely, the residual This rank reduction process in linear algebra is known as deflation. The most expensive part of this algorithm is in Line (a), as it requires  selecting the maximum element of it in modulus in Line (a) and it costs $\mathcal{O}(mn)$, for a matrix of size $m\times n$. Also, in Line (e), the residual is computed, which is expensive.

The mathematical elegance of ACA lies in its residual update formula, which avoids constructing the full residual matrix. Instead, we look for only the maximum in the modulus of a row or a column. This leads to reduction of computational complexity from $\mathcal{O}(mn)$ to $\mathcal{O}(m+n)$. More precisely, each residual row or column is computed on demand using the current approximation:
\begin{eqnarray}
r(i_k) = R^{(k)}(i_k,:) - \sum_{t=1}^{k-1} U(i_k,t) V(:,t)^T,\\
r(j_k) = R^{(K)}(:,j_k) - \sum_{t=1}^{k-1} U(:,t) V(j,t)^T,    
\end{eqnarray}
depending on, row or column indices should be selected, respectively. This formula reveals that the algorithm maintains a rank-$k$ decomposition throughout its execution, and each iteration simply augments the existing approximation by one additional rank.

Several theoretical properties of classical ACA are worth noting \cite{bebendorf2000approximation,bebendorf2003adaptive,Tyrtyshnikov2000,Savostyanov2006}. First, the algorithm is guaranteed to converge for matrices that are numerically low-rank, with the convergence rate depending on the singular value decay of the matrix. Second, the pivot selection strategy based on maximum residual entry is not merely heuristic; it can be shown that the resulting approximation is quasi-optimal in the sense that the error is bounded by a constant multiple of the best rank-$k$ approximation error. Third, the computational complexity scales as $\mathcal{O}(k(m+n))$ per iteration, making it substantially more efficient than full singular value decomposition for large matrices.

\begin{figure}[htbp]
    \centering
\includegraphics[width=0.8\textwidth]{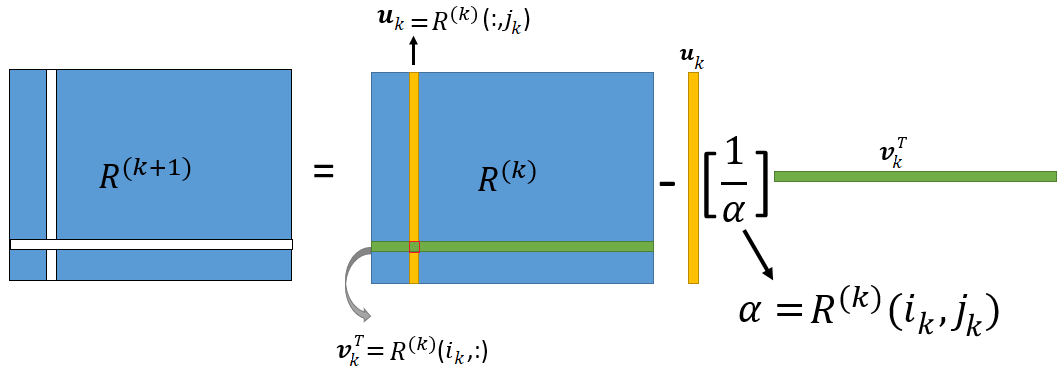}
    \caption{The deflation process in the ACA method is where, at each iteration, one column and one row corresponding to the selected pivot $(i_k,j_k)$ are zeroed out. The while column and row in the right-hand-side matrix mean a zero column and a zero row.}
    \label{fig:def}
\end{figure}

\subsection{Native Quaternion ACA Algorithm}\label{sub:NQACA}
Extending ACA to the quaternion algebra requires careful handling of non-commutativity. The direct generalization is non-trivial because the order of multiplication matters in $\mathbb{H}$, and the standard column normalization $V(:,k) = V(:,k) / V(i_k,j_k)$ must be reinterpreted in terms of quaternion multiplication. 

Our native quaternion ACA implementation, denoted QACA, follows the algorithmic structure presented in Algorithm~\ref{alg:quaternion_aca}, see the appendix, which describes the practical MATLAB implementation using the QTFM toolbox, covering quaternion array handling, arithmetic, normalization, pivot selection, and optimizations for function handles. This is analogous to the standard ACA for complex matrices, but adapted to the non-commutative quaternion setting. The algorithm constructs a rank-$k$ approximation $A \approx U V^T$ with $U \in \mathbb{H}^{m \times k}$ and $V \in \mathbb{H}^{n \times k}$, where $V^T$ denotes the quaternion  transpose. The main steps are as follows:

\paragraph{Initialization.}
The algorithm begins by pre-allocating quaternion factor matrices $U$ and $V$ with dimensions corresponding to the maximum allowed rank $r_{\max}$:
\[
U = \mathbf{0}\in\mathbb{H}^{m \times r_{\max}}, \quad V = \mathbf{0}\in\mathbb{H}^{n \times r_{\max}}.
\]
This pre-allocation strategy avoids dynamic resizing during the iterative process and improves performance. The Frobenius norm of the matrix $A_{\text{norm}} = \|A\|_F$ is computed for use in the stopping criterion. If $A_{\text{norm}} = 0$, the algorithm returns empty factors.

\paragraph{Initial Row Selection.}
Unlike classical ACA which typically selects the first pivot using the global maximum entry, our implementation finds the first non-zero row:
\[
i = \min\{r : \|A(r,:)\| > 0\}.
\]
This approach ensures robustness for matrices with structured sparsity patterns and avoids potential issues with zero matrices. The selected row is marked as used by maintaining boolean arrays of selected rows and columns. It is also possible to select the first row randomly. In this case, if a zero row is selected, one can randomly sample a set of \(m+n\) points from the data matrix to check whether the entire matrix is zero. This sampling step is not performed over the whole matrix to find the maximum of the residual, as previously mentioned, due to the high computational cost, (see Lines 10-15) of Algorithm \ref{alg:quaternion_aca}. Note that the MATLAB function ``randi'', generates uniformly distributed random integers.

\begin{algorithm}[H]
\caption{Native Quaternion Adaptive Cross Approximation (QACA)}
\label{alg:quaternion_aca}
\begin{algorithmic}[1]
\Input{$A \in \mathbb{H}^{m \times n}$ (target matrix), $\varepsilon > 0$ (tolerance), $r_{\max}$ (max rank)}
\Output{$U \in \mathbb{H}^{m \times k}$, $V \in \mathbb{H}^{n \times k}$ with $\widetilde{A}=A \approx U V^T$, rank $k$, relative error $err$}
\State $U \gets \mathbf{0}_{m \times r_{\max}}$, $V \gets \mathbf{0}_{n \times r_{\max}}$
\State $A_{\text{norm}} \gets \|A\|_F$; \textbf{if} $A_{\text{norm}} = 0$ \textbf{then return} $U,V \gets \emptyset$
\State rows $\gets \emptyset$, cols $\gets \emptyset$
\State $i_1 \gets \min\{p : \|A(p,:)\| > 0\}$; rows $\gets$ rows $\cup \{i_1\}$  
\For{$k = 1$ to $r_{\max}$}
    \State $r \gets A(i_k,:) - U(i_k,1:k-1) V(:,1:k-1)^T$  \Comment{Residual row}
    \State $j_k \gets \arg\max_{t \notin \text{cols}} |r(t)|$
    \If{$|r(j_k)| \le \varepsilon A_{\text{norm}}$}
    
    \State \(\mathcal{S} = \{(i_s, j_s) := (\text{randi}(m), \text{randi}(n))\}\), \(\#\mathcal{S} \sim m + n\)
    \State  \(\widetilde{A} := U(:,1:k-1)V(:,1:k-1)^T \); 
    \State \(\text{err}_\mathcal{S}^2 := \sum_{\mathcal{S}} |(A - \widetilde{A})(i_s, j_s)|^2\), \(\text{nrm}_\mathcal{S}^2 := \sum_{\mathcal{S}} |A(i_s, j_s)|^2\)
    \State \((i_k, j_k) := \max_{\mathcal{S}}(\text{abs}(A - \widetilde{A}))\)
    \State \textbf{if} \(\text{err}_\mathcal{S} \le \varepsilon \text{nrm}_\mathcal{S}\) \textbf{then return} \( U,\, V\) of rank \(r := k\).
    \textbf{break} \EndIf
    \State $\alpha \gets r(j_k)$; 
    \State $c \gets A(:,j_k) - U(:,1:k-1) V(j,1:k-1)^T$  \Comment{Residual column}
    \State $U(:,k) \gets c$; $V(:,k) \gets \alpha^{-1} r^T$
    \State $i \gets \arg\max_{t \notin \text{rows}} |c(t)|$; rows $\gets$ rows $\cup \{i_k\}$; cols $\gets$ cols $\cup \{j_k\}$
\EndFor
\State $U \gets U(:,1:k)$, $V \gets V(:,1:k)$ \textbf{if} $k>0$
\State $err \gets \|A - U V^T\|_F / A_{\text{norm}}$ \textbf{if} $k>0$ \textbf{else} $err \gets 1$
\State \Return $U, V,$ rank$(k)$, error$(err)$
\end{algorithmic}
\end{algorithm}

\paragraph{Iterative Rank-One Updates.}
At each iteration $k = 1, 2, \ldots, r_{\max}$, the algorithm performs the following operations:

\subparagraph{Residual Row Computation.}
The residual row at the current pivot row $i$ is computed as:
\[
r = A(i,:) - U(i,1:k-1) V(:,1:k-1)^T,
\]
where $r \in \mathbb{H}^{1 \times n}$. The term $U(i,1:k-1) \in \mathbb{H}^{1 \times (k-1)}$ and $V(:,1:k-1)^T \in \mathbb{H}^{(k-1) \times n}$ are multiplied using quaternion arithmetic, with the transpose applied to the $V$ factor. This residual row represents the error in approximating the $i$-th row of $A$ using the current $k-1$ rank approximation.

\subparagraph{Pivot Column Selection.}
The pivot column $j$ is selected from the set of unselected columns as:
\[
j = \arg\max_{j \notin \text{cols}} |r(j)|.
\]
This strategy focuses on columns where the current approximation is poorest, ensuring that the algorithm captures the dominant residual features. The selection is governed by the stopping criterion:
\[
|r(j)| \leq \varepsilon \cdot A_{\text{norm}},
\]
where $\varepsilon$ is the user-specified tolerance. Additionally, the pivot value $\alpha = r(j)$ is checked for near-zero magnitude, the algorithm check whether we have selected a zero row or not and depending  on that it terminates to avoid numerical instability or continue by selecting a new pivots using the sampling a part of the matrix components, see Lines 10-16 in Algorihm \ref{alg:quaternion_aca}.

\subparagraph{Residual Column Computation.}
The residual column at the selected pivot column $j$ is computed as:
\[
c = A(:,j) - U(:,1:k-1) V(j,1:k-1)^T,
\]
where $c \in \mathbb{H}^{m \times 1}$. This column represents the error in approximating the $j$-th column of $A$.

\subparagraph{Update and Normalization.}
The new rank-one contribution is constructed as:
\[
U(:,k) = c, \quad V(:,k) = \alpha^{-1} r^T,
\]
where $\alpha^{-1}$ is the quaternion inverse of $\alpha$. This normalization is critical because it enforces the condition $U(i,k) = c(i) = 1$ after scaling, following from the fact that $r(j) = \alpha$.

\subparagraph{Next Row Selection.}
The next pivot row $i$ is selected from the unselected rows based on the residual column:
\[
i = \arg\max_{i \notin \text{rows}} |c(i)|.
\]
This adaptive row selection strategy, coupled with the column selection approach, ensures that the algorithm systematically covers the most significant residual entries. Both the selected row and column are marked as used.

\paragraph{Post-Processing.}
% Upon termination, the algorithm removes unused columns from the factors:
% \[
% U = U(:,1:k_{\text{actual}}), \quad V = V(:,1:k_{\text{actual}}),
% \]
% where $k_{\text{actual}}$ is the actual rank achieved. 
The relative approximation error is computed as:
\[
err = \frac{\|A - U V^T\|_F}{\|A\|_F}.
\]
If no rank was achieved ($k_{\text{actual}} = 0$), the error is set to $err = 1$.

\paragraph{Termination Criteria.}
The algorithm employs multiple stopping criteria to ensure robustness and efficiency:
\begin{enumerate}
    \item \textbf{Maximum rank reached:} $k > r_{\max}$.
    \item \textbf{Tolerance satisfied:} $|r(j)| \leq \varepsilon \cdot A_{\text{norm}}$.
    %\item \textbf{Near-zero pivot:} $|\alpha| < 10^{-15}$.
    \item \textbf{No unselected rows or columns remaining.}
\end{enumerate}

This comprehensive set of stopping criteria ensures that the algorithm terminates gracefully under various conditions, including rank-deficient matrices, ill-conditioned pivots, or when the desired accuracy is achieved.

The native quaternion ACA algorithm preserves several important properties of classical ACA:
\begin{enumerate}
    \item \textbf{Quasi-optimality:} The approximation error is bounded by a constant multiple of the best rank-$k$ quaternion approximation error.
    \item \textbf{Computational efficiency:} The complexity per iteration is $\mathcal{O}(k(m+n))$ quaternion operations, making it competitive with its real counterpart.
    \item \textbf{Memory efficiency:} Only the selected rows and columns of $A$ are accessed, avoiding the need to store the full matrix explicitly.
    \item \textbf{Non-commutative algebra:} The algorithm correctly handles the non-commutativity of quaternion multiplication through the use of  transpose and proper multiplication order.
\end{enumerate}

The complete implementation of QACA is provided in the appendix %\url{https://github.com/SalmanAhmadi-Asl/Quaternion-ACA.git}, 
demonstrating how these algorithmic principles are realized in practice using the QTFM toolbox for MATLAB. The codes are presented in the appendix.

Algorithm \ref{alg:quaternion_aca} has a parameter called maximum rank, which must be provided to the algorithm as an input. However, estimating this parameter in real scenarios is difficult. To make the algorithm fully adaptive and eliminate the need for this parameter, Algorithm \ref{alg:quaternion_aca} can be modified into Algorithm \ref{alg:qaca_auto}, in which the algorithm automatically determines the rank for a given prescribed approximation bound $\epsilon$. This algorithm was initially proposed in \cite{Savostyanov2006} and \cite{Tyrtyshnikov2000} for low-rank approximation of real-valued matrices.

\begin{algorithm}
\caption{Native Quaternion ACA with Automatic Rank Detection}\label{alg:qaca_auto}
\begin{algorithmic}[1]
\Input Matrix \(A = [A(i, j)]\in\mathbb{H}^{m\times  n}\), accuracy threshold \(\varepsilon\).
\Output Rank-\(r\) matrix \(\widetilde{A} = U V^T = \sum_{k=1}^r u_k v_k^T\) such that \(\|A - \widetilde{A}\|_F \le \varepsilon \|A\|_F\).
\Initialization \(U = 0\), \(V = 0\), \(\widetilde{A} = 0\); \(\text{nrm} = 0\); \(\mathcal{I} = \mathcal{J} = \emptyset\); \(k = 1\), \(j_1 = \text{randi}(q)\)
\State \(u_k := (A - \widetilde{A})(:, j_k) = A(:, j_k) - U V(j_k, :)^T\); \(i_k := \max(\text{abs}(u_k))\)
\State \textbf{if} \(u_k(i_k) = 0\) \textbf{goto} step 6
\State \(\hat{v}_k^T := (A - \widetilde{A})(i_k, :) = A(i_k, :) - U(i_k, :) V^T\); \(v_k := \hat{v}_k / \hat{v}_k(j_k)\); \(j_{k+1} := \max_{j \neq j_k}(\text{abs}(v_k))\)
\State \(\text{nrm}^2 := \text{nrm}^2 + 2(U^T u_k, V^T v_k) + \|u_k\|_2^2 \|v_k\|_2^2\); \(\text{err}^2 := (\min(p, q) - k) \|u_k\|_2^2 \|v_k\|_2^2\)
\State \(\widetilde{A} := \widetilde{A} + u_k v_k^T = [U u_k] [V v_k]^T\); \(U := [U u_k]\), \(V := [V v_k]\); \(\mathcal{I} := \mathcal{I} \cup \{i_k\}\), \(\mathcal{J} := \mathcal{J} \cup \{j_k\}\)
\State \{NB: Now \(\widetilde{A}\) interpolates \(A\) on set of rows \((\mathcal{I}, :)\) and set of columns \((:, \mathcal{J})\)\}
\State \textbf{if} \(\text{err} \le \varepsilon \text{nrm}\) \textbf{then}
    \State \(\mathcal{S} = \{(i_s, j_s) := (\text{randi}(p), \text{randi}(q))\}\), \(\#\mathcal{S} \sim m + n\)
    \State \(\text{err}_\mathcal{S}^2 := \sum_{\mathcal{S}} |(A - \widetilde{A})(i_s, j_s)|^2\), \(\text{nrm}_\mathcal{S}^2 := \sum_{\mathcal{S}} |A(i_s, j_s)|^2\)
    \State \((i_{k+1}, j_{k+1}) := \max_{\mathcal{S}}(\text{abs}(A - \widetilde{A}))\)
    \State \textbf{if} \(\text{err}_\mathcal{S} \le \varepsilon \text{nrm}_\mathcal{S}\) \textbf{then return} \(\widetilde{A} = U V^T\) of rank \(r := k\).
\State \textbf{end if}
\State \(k := k + 1\), \textbf{goto} step 1
\end{algorithmic}
\end{algorithm}

When implementing ACA in quaternion arithmetic, {\it every multiplication must respect left-to-right order}. The update
\[
V_k \;\longleftarrow\; \alpha^{-1} \, r^T
\]
is not optional; it is mathematically mandated by the non-commutative structure of $\mathbb{H}$. Replacing it with $V_k \;\longleftarrow\; r^T \alpha^{-1}$ destroys the algebraic consistency of the rank-1 update and invalidates the algorithm's error guarantees. To be more precise, let us explain it in more details. To do so, we first discuss the rank-1 update in the ACA algorithm. Given a matrix \(A\in\mathbb{H}^{m\times n}\), let the selected row and column indices be \(i_k\) and \(j_k\), respectively. For the rank-1 update
\begin{equation}\label{for_update}
R=A- A(:,j_k)\frac{1}{\alpha} A(i_k,:),
\end{equation}
where \(\alpha=A(i_k,j_k)\) is the pivot, we can easily see that
\begin{align}
R(i_k,:)&=A(i_k,:)-A(i_k,j_k)\frac{1}{\alpha} A(i_k,:)=0,\\
R(:,j_k)&=A(:,j_k)-A(:,j_k)\frac{1}{\alpha} A(i_k,j_k)=0.
\end{align}
This means that such a rank-1 update simultaneously zeros out a row and a column of $R$. However, due to the non-commutativity of quaternions, the following rank-1 updates
\begin{equation}\label{defla_2}
R=A- A(:,j_k)A(i_k,:)\frac{1}{\alpha},
\end{equation}
and
\begin{equation}\label{defla_3}
R=A-\frac{1}{\alpha}A(:,j_k)A(i_k,:),
\end{equation}
are not the same as \eqref{for_update}. In \eqref{defla_2}, only the \(j_k\)-th column will be zero, while in \eqref{defla_3}, only the \(i_k\)-th row will be zero. Thus, unlike real or complex matrices, the way the factors are scaled in the QACA method matters.

\begin{remark}
Please note that normalization can be performed using the factor \(U\), but in this case, the multiplication must be applied from the right. More precisely,
\[
U_k \longleftarrow c\,\alpha^{-1}.
\]
As discussed above, if instead we use  
\[
U_k \longleftarrow \alpha^{-1}\,c,
\]
then at each iteration only one row becomes zero, not a column. Similarly, using  
\[
V_k \longleftarrow r^T \, \alpha^{-1}
\]
zeros out only one column, and not  a row. However, when we apply either  
\[
V_k \longleftarrow \alpha^{-1}\,r^T
\]
or  
\[
U_k \longleftarrow c\,\alpha^{-1},
\]
both one row and one column are zeroed simultaneously. 
\end{remark}

To better understand the algorithm, in the following we provide one example for a matrix of size $3\times 3$ and right quaternion rank 2. Let
\[
a_1 = \begin{bmatrix}1\\ \mathbf{i}\\ 0\end{bmatrix},\quad
a_2 = \begin{bmatrix}0\\ \mathbf{j}\\ 1\end{bmatrix}
\in \mathbb{H}^{3\times 1},
\qquad
b_1 = \begin{bmatrix}1\\ 0\\ \mathbf{k}\end{bmatrix},\quad
b_2 = \begin{bmatrix}0\\ 1\\ \mathbf{i}\end{bmatrix}
\in \mathbb{H}^{3\times 1}.
\]
Define
\[
A = a_1 b_1^{T} + a_2 b_2^{T} =\begin{bmatrix}
1 & 0 & \mathbf{k}\\
\mathbf{i} & \mathbf{j} & -\mathbf{j}-\mathbf{k}\\
0 & 1 & \mathbf{i}
\end{bmatrix}\in \mathbb{H}^{3\times 3},
\]
so , \(A\) is exactly rank \(2\) by construction. Algorithm \ref{alg:quaternion_aca} select the first row as it is the first nonzero row, so $i_1=1$ and 
\[
V(:,1)^T=\begin{bmatrix}
    1& 0& \mathbf{k}
\end{bmatrix}.
\]
Then, the maximum of the elements of this row in modulus is selected to be the selected column. Since $|1|=|\mathbf{k}|=1$, the algorithm selects the first column, so $j_1=1$ and $U(:,1)=\begin{bmatrix}
    1\\\mathbf{i}\\0
\end{bmatrix}$, and the pivot is $A(i_1,j_1)=1$. So, the rank-1 deflation becomes

\begin{eqnarray}
R^{(2)}&=&R^{(1)}-U(:,1)\frac{1}{A(i_1,j_1)}V(:,1)^T=
\begin{bmatrix}
1 & 0 & k\\
\mathbf{i} & \mathbf{j} & -\mathbf{j}-\mathbf{k}\\
0 & 1 & \mathbf{i}
\end{bmatrix}-\begin{bmatrix}
    1\\\mathbf{i}\\0
\end{bmatrix}\frac{1}{1}\begin{bmatrix}
    1& 0& \mathbf{k}
\end{bmatrix},\\
&=&\begin{bmatrix}
1 & 0 & k\\
\mathbf{i} & \mathbf{j} & -\mathbf{j}-\mathbf{k}\\
0 & 1 & \mathbf{i}
\end{bmatrix}-\begin{bmatrix}
1 & 0 & k\\
\mathbf{i} & 0 & -\mathbf{j}\\
0 & 0 & 0
\end{bmatrix}=
\begin{bmatrix}
0 & 0 & 0\\
0 & \mathbf{j} & -\mathbf{k}\\
0 & 1 & \mathbf{i}
\end{bmatrix},
\end{eqnarray}
where $A=R^{(1)}$. As can be seen, the first row and the first column become zero, or more precisely the approximation $A\approx U(:,1)V(:,1)^T$, interpolates the matrix $A$, in the first row and the first column.

In the second stage, we select another row index by searching for the maximum of an element of $V(:,1)^T$ in modulus, which is $i_2=2$, because $i_2$ cannot be 1 as was selected in the first stage and $|V(2,1)^T|=|\mathbf{i}|=1$ and $|V(3,1)^T|=|\mathbf{i}|=0$. To select the second column index, now $k=1$ and the maximum in the module of the residual
\[
r(i_2)=A(i_2,:)-U(i_2,1)V(:,1)^T,
\]
or
\[
r(2)=A(2,:)-U(2,1)V(:,1)^T=\begin{bmatrix}
\mathbf{i}&\mathbf{j}&-\mathbf{j}-\mathbf{k}
\end{bmatrix}-(\mathbf{i})
\begin{bmatrix}
1&0&\mathbf{k}
\end{bmatrix}=\begin{bmatrix}
0&\mathbf{j}&-\mathbf{k}
\end{bmatrix}.
\]
is computed. Here $j_2=2$ because $j_2$ can be 1 as was already selected and $|\mathbf{j}|=|\mathbf{k}|=1$. So, we have 
\[
U(:,i_2)=\begin{bmatrix}
0\\\mathbf{j}\\1    
\end{bmatrix},\quad  V(:,j_2)^T=\begin{bmatrix}
0&\mathbf{j}&-\mathbf{k}   
\end{bmatrix},
\]
with the second pivot $A(i_2,j_2)=\mathbf{j}$. The second deflation process will be 
\begin{eqnarray}
R^{(3)}&=&R^{(2)}-U(:,i_2)\frac{1}{A(i_2,j_2)}V(:,j_2)^T=\begin{bmatrix}
0 & 0 & 0\\
0 & \mathbf{j} & -\mathbf{k}\\
0 & 1 & \mathbf{i}
\end{bmatrix}-\begin{bmatrix}
0\\\mathbf{j}\\1    
\end{bmatrix}
\frac{1}{\mathbf{j}}
\begin{bmatrix}
0&\mathbf{j}&-\mathbf{k}   
\end{bmatrix},\\
&=&\begin{bmatrix}
0 & 0 & 0\\
0 & \mathbf{j} & -\mathbf{k}\\
0 & 1 & \mathbf{i}
\end{bmatrix}-
\begin{bmatrix}
0 & 0 & 0\\
0 & \mathbf{j} & -\mathbf{k}\\
0 & 1 & \mathbf{i}
\end{bmatrix}=\begin{bmatrix}
0&0&0\\
0&0&0\\
0&0&0
\end{bmatrix}.
\end{eqnarray}
As expected, the new approximation $A\approx U(:,1:2)V(:,1:2)^T$, became exact as the original matrix has rank 2 and it also interpolated the original matrix in the row indices $i_1=1,\,i_2=2$ and column indices $j_1=1,\,j_2=1$. 

In the next section, we provide more theoretical aspects of the proposed adaptive cross approximation algorithm in terms of rank reduction property and approximation accuracy.

\section{Theoretical Analysis}\label{sec:theory}
In this section, we present the convergence and error bounds for the proposed algorithm.

\begin{theorem}[Deflation Property of QACA]
\label{thm:qaca-deflation}
Let $A\in\mathbb{H}^{m\times n}$
and suppose that $\operatorname{rank}_{\mathbb{H}}(A)=r.
$
Let
$R^{(0)}=A.
$
Suppose that at the $k$th successful step of QACA a nonzero pivot
$\alpha_k=R^{(k)}_{i_kj_k}\neq0
$
is selected and the residual is updated according to
\[
R^{(k+1)}
=
R^{(k)}
-
R^{(k)}_{:,j_k}
\alpha_k^{-1}
R^{(k)}_{i_k,:}.
\]
Then
\[
{
\operatorname{rank}_{\mathbb{H}}
\left(R^{(k+1)}\right)
=
\operatorname{rank}_{\mathbb{H}}
\left(R^{(k)}\right)-1.
}
\]
Consequently, if QACA performs $r$ successful nonzero-pivot steps,
then $R^{(r)}=0,$
and hence $A=U_rV_r^T.$ Here $\mathbb{H}^m$ is regarded as a right quaternion vector space.
\end{theorem}

\begin{proof}
Fix a successful step and, for simplicity of notation, write
\[
R=R^{(k)},\qquad
\alpha=R_{ij}\neq0,
\]
where $i=i_k,\,j=j_k.$
Define
\[
c=R_{:j}\in\mathbb{H}^{m\times1},
\qquad
r=R_{i,:}\in\mathbb{H}^{1\times n},
\]
and $u=c\alpha^{-1}.$
Since $u_i=c_i\alpha^{-1}
=
R_{ij}R_{ij}^{-1}
=
1,$
we have $e_i^Tu=1.$
The residual update becomes
\[
\begin{aligned}
R^{(k+1)}
&=
R-c\alpha^{-1}r\\
&=
R-ue_i^TR\\
&=
(I_m-ue_i^T)R.
\end{aligned}
\]
Define $P=I_m-ue_i^T.$ Thus $
R^{(k+1)}=PR.$ We regard $\mathbb{H}^m$ as a right vector space over
$\mathbb{H}$. The matrix $P$ defines a right-linear transformation
because, for every $x\in\mathbb{H}^m$ and $q\in\mathbb{H}$,
\[
P(xq)
=
(I_m-ue_i^T)(xq)
=
Px\,q.
\]
We next determine the kernel of $P$. First,
\[
Pu
=
u-ue_i^Tu
=
u-u
=
0.
\]
Hence $\operatorname{span}_{\mathbb H}^{R}\{u\}
\subseteq
\ker_{\mathbb H}(P).$ Conversely, suppose that $Px=0,$
then $x-ue_i^Tx=0,$
and therefore $x=u(e_i^Tx).$
Since $e_i^Tx\in\mathbb H$, this shows that
\[
x\in
\operatorname{span}_{\mathbb H}^{R}\{u\}.
\]
Consequently,
\[
{
\ker_{\mathbb H}(P)
=
\operatorname{span}_{\mathbb H}^{R}\{u\}.
}
\]
Since $u\neq0$ because $u_i=1$, we obtain $\dim_{\mathbb H}\ker(P)=1.$
Now observe that
\[
u=c\alpha^{-1}
=
R_{:j}\alpha^{-1}.
\]
Since $R_{:j}\in\mathcal{C}_{R}(R)$ and the column space is a right
quaternion subspace, $u\in\mathcal{C}_{R}(R).$
Therefore, $\operatorname{span}_{\mathbb H}^{R}\{u\}
\subseteq
\mathcal{C}_{R}(R).$
Hence
\[
\mathcal{C}(R)\cap\ker(P)
=
\operatorname{span}_{\mathbb H}^{R}\{u\},
\]
and consequently
\[
\dim_{\mathbb H}^{R}
\left(
\mathcal{C}_{R}(R)\cap\ker(P)
\right)
=1.
\]
Since $R^{(k+1)}=PR,$ we have $\mathcal{C}_{R}(R^{(k+1)})
=
P\mathcal{C}_{R}(R).$
Consider the restriction
\[
P|_{\mathcal{C}_{R}(R)}
:
\mathcal{C}_R(R)\longrightarrow
\mathcal{C}_{R}(R^{(k+1)}).
\]
Its kernel is
\[
\ker\left(P|_{\mathcal{C}_{R}(R)}\right)
=
\mathcal{C}_{R}(R)\cap\ker(P),
\]
which has right quaternion dimension one. By the rank--nullity theorem for finite-dimensional right vector
spaces over a division ring,
\[
\dim_{\mathbb H}\mathcal{C}_{R}(R)
=
\dim_{\mathbb H}
\ker\left(P|_{\mathcal{C}_{R}(R)}\right)
+
\dim_{\mathbb H}
P\mathcal{C}_{R}(R).
\]
Therefore,
\[
\begin{aligned}
\operatorname{rank}_{\mathbb H}(R^{(k+1)})
&=
\dim_{\mathbb H}\mathcal{C}_{R}(R^{(k+1)})\\
&=
\dim_{\mathbb H}\mathcal{C}_{R}(R)-1\\
&=
\operatorname{rank}_{\mathbb H}(R)-1.
\end{aligned}
\]
Thus
\[
{
\operatorname{rank}_{\mathbb H}(R^{(k+1)})
=
\operatorname{rank}_{\mathbb H}(R^{(k)})-1.
}
\]
Starting from
\[
\operatorname{rank}_{\mathbb H}(R^{(0)})=r,
\]
induction gives
\[
\operatorname{rank}_{\mathbb H}(R^{(k)})=r-k,
\qquad
k=0,2,\ldots,r.
\]
In particular,
\[
\operatorname{rank}_{\mathbb H}(R^{(r)})=0.
\]
The only quaternion matrix of rank zero is the zero matrix, and hence
\[
R^{(r)}=0.
\]
Therefore
\[
A=U_rV_r^T.
\]
This implies that the native quaternion ACA algorithm recovers an exact factorization of any rank-$r$ matrix $A \in \mathbb{H}^{m \times n}$ after at most $r$ iterations, assuming exact arithmetic and nonzero pivots.
\end{proof}

\begin{theorem}[CUR Representation of the QACA Approximation]
\label{thm:qaca-cur-equivalence}
Let $A\in\mathbb{H}^{m\times n},
$ and suppose that the QACA
algorithm performs $k$ successful steps with nonzero pivots
\[
\alpha_\ell
=
R^{(\ell-1)}_{i_\ell j_\ell}
\neq 0,
\qquad
\ell=1,2,\ldots,k.
\]
Let
\[
I_k=\{i_1,2,\ldots,i_k\},
\qquad
J_k=\{j_1,2,\ldots,j_k\}
\]
denote the selected row and column indices, respectively. Define
the selected column matrix, selected row matrix, and selected cross
matrix by
\[
C_k=A(:,J_k)\in\mathbb{H}^{m\times k},
\]
\[
R_k=A(I_k,:)\in\mathbb{H}^{k\times n},
\]
and
\[
W_k=A(I_k,J_k)\in\mathbb{H}^{k\times k}.
\]
Assume that $W_k$ is nonsingular. Then the rank-$k$ approximation generated by QACA can be written as
\[
{
A_k=U_kV_k^T
=
C_kW_k^{-1}R_k.
}
\]
Consequently, QACA produces an implicit quaternion CUR
(or skeleton) decomposition
\[
{
A_k=C_kM_kR_k,
\qquad
M_k=W_k^{-1}
=
A(I_k,J_k)^{-1}.
}
\]
In particular, the middle factor in the CUR representation is the
inverse of the selected cross matrix. It should not be confused with
the QACA factor $U_k$\footnote{In this theorem, \(R^{(l)}\) (with a superscript) denotes a residual matrix, whereas \(R_k\) (with a subscript) denotes the factor containing the selected rows ($k$ rows) in the CUR approximation. Please be careful to distinguish between the two notations.}. 
\end{theorem}
 
\begin{proof}
We prove the result by induction on the number of successful QACA
steps. For convenience, denote by $A_\ell=U_\ell V_\ell^T$
the approximation after $\ell$ successful QACA steps, and let
\[
R^{(\ell)}=A-A_\ell
\]
be the corresponding residual. At step $\ell$, QACA selects a pivot $(i_\ell,j_\ell)$
with $
\alpha_\ell
=
R^{(\ell-1)}_{i_\ell j_\ell}
\neq0.$ Define
\[
c_\ell=R^{(\ell-1)}_{:,j_\ell},
\qquad
r_\ell=R^{(\ell-1)}_{i_\ell,:}.
\]
The QACA update is
\begin{equation}\label{EQ:1}
A_\ell
=
A_{\ell-1}
+
c_\ell\alpha_\ell^{-1}r_\ell.
\end{equation}
The order of the quaternion factors in \eqref{EQ:1} is essential.

\medskip
\noindent
\textbf{Base case: $k=1$.} For the first selected row and column, we have
\[
C_1=A(:,j_1)=c_1,\quad 
R_1=A(i_1,:)=r_1,
\]
and
\[
W_1=A(i_1,j_1)=\alpha_1.
\]
Hence
\[
C_1W_1^{-1}R_1
=
c_1\alpha_1^{-1}r_1.
\]
By the first QACA update, $A_1=c_1\alpha_1^{-1}r_1.$ Therefore ${
A_1=C_1W_1^{-1}R_1.
}$

\medskip
\noindent
\textbf{Induction hypothesis.} Assume that after $k$ successful steps,
\begin{equation}\label{EQ:2}
{
A_k=C_kW_k^{-1}R_k,
}
\end{equation}
where
\[
C_k=A(:,J_k),\qquad
R_k=A(I_k,:),\qquad
W_k=A(I_k,J_k).
\]
We prove that
\[
A_{k+1}
=
C_{k+1}W_{k+1}^{-1}R_{k+1}.
\]
Let the new row and column selected by QACA be
\[
i=i_{k+1},
\qquad
j=j_{k+1}.
\]
For notational convenience, define
\[
p=A(I_k,j),
\qquad
q=A(i,J_k),
\qquad
a=A(i,j).
\]
Then the enlarged cross matrix is
\[
W_{k+1}
=
\begin{bmatrix}
W_k & p\\
q & a
\end{bmatrix}.
\]
The residual column selected by QACA is
\[
c_{k+1}
=
A(:,j)-A_k(:,j).
\]
Using the induction hypothesis \eqref{EQ:2},
\begin{equation}\label{EQ:3}
c_{k+1}
=
A(:,j)-C_kW_k^{-1}p.
\end{equation}
Similarly, the residual row is
\[
r_{k+1}
=
A(i,:)-A_k(i,:),
\]
and therefore
\begin{equation}\label{EQ:4}
r_{k+1}
=
A(i,:)-qW_k^{-1}R_k.
\end{equation}
The new pivot is $\alpha_{k+1}
=
R^{(k)}_{ij}.$
Using \eqref{EQ:3}, evaluated at row $i$,
\[
\alpha_{k+1}
=
a-qW_k^{-1}p.
\]
Thus $\alpha_{k+1}$ is precisely the Schur complement of $W_k$
in $W_{k+1}$. Since $W_k$ and $\alpha_{k+1}$ are nonsingular, the block inverse
formula gives
\begin{equation}\label{EQ:5}
W_{k+1}^{-1}
=
\begin{bmatrix}
W_k^{-1}
+
W_k^{-1}p\alpha_{k+1}^{-1}qW_k^{-1}
&
-W_k^{-1}p\alpha_{k+1}^{-1}
\\[2mm]
-\alpha_{k+1}^{-1}qW_k^{-1}
&
\alpha_{k+1}^{-1}
\end{bmatrix}.
\end{equation}
The enlarged selected column and row matrices are
\[
C_{k+1}
=
\begin{bmatrix}
C_k & A(:,j)
\end{bmatrix},
\]
and
\[
R_{k+1}
=
\begin{bmatrix}
R_k\\
A(i,:)
\end{bmatrix}.
\]
Substituting \eqref{EQ:5} and multiplying while preserving the order of
quaternion multiplication yields
\[
\begin{aligned}
C_{k+1}W_{k+1}^{-1}R_{k+1}
={}&
C_kW_k^{-1}R_k
\\
&+
\left(A(:,j)-C_kW_k^{-1}p\right)
\alpha_{k+1}^{-1}
\left(A(i,:)-qW_k^{-1}R_k\right).
\end{aligned}
\]
Using \eqref{EQ:3} and \eqref{EQ:4}, this becomes
\[
C_{k+1}W_{k+1}^{-1}R_{k+1}
=
C_kW_k^{-1}R_k
+
c_{k+1}\alpha_{k+1}^{-1}r_{k+1}.
\]
By the induction hypothesis,
\[
C_kW_k^{-1}R_k=A_k.
\]
Hence
\[
C_{k+1}W_{k+1}^{-1}R_{k+1}
=
A_k+c_{k+1}\alpha_{k+1}^{-1}r_{k+1}.
\]
By the QACA update \eqref{EQ:1},
\[
A_{k+1}
=
A_k+c_{k+1}\alpha_{k+1}^{-1}r_{k+1}.
\]
Therefore
\[
{
A_{k+1}
=
C_{k+1}W_{k+1}^{-1}R_{k+1}.
}
\]
By induction, for every $k$ for which the selected cross matrix
$W_k$ is nonsingular,
\[
{
U_kV_k^T
=
A_k
=
C_kW_k^{-1}R_k.
}
\]
Defining $M_k=W_k^{-1},$ we obtain the QCUR representation
\[
{
A_k=C_kM_kR_k,
\qquad
M_k=A(I_k,J_k)^{-1}.
}
\]
This completes the proof.
\end{proof}

The theorem immediately yields the following important consequences.

\begin{corollary}[Interpolation Property]
\label{cor:interpolation}
Under the same assumptions, the residual $R^{(k)}$ vanishes on the selected rows and columns:
\[
(R^{(k)})_{\mathcal{I},:} = 0, \qquad (R^{(k)})_{:,\mathcal{J}} = 0.
\]
\end{corollary}

\begin{proof}
From Theorem \ref{thm:qaca-cur-equivalence}, for any $i \in \mathcal{I}$ and any $j$:
\[
(R^{(k)})_{i,j} = A_{i,j} - (A_{:,\mathcal{J}} P^{-1} A_{\mathcal{I},:})_{i,j}.
\]
Considering $P=A_{\mathcal{I},\mathcal{J}}$, we have
\begin{eqnarray}
(R_k)_{\mathcal{I},:} = A_{\mathcal{I},:} - A_{\mathcal{I},\mathcal{J}} A_{\mathcal{I},\mathcal{J}}^{-1} A_{\mathcal{I},:}=0,\\
(R_k)_{:,\mathcal{J}} = A_{:,\mathcal{J}} - A_{:,\mathcal{J}} A_{\mathcal{I},\mathcal{J}}^{-1} A_{\mathcal{I},\mathcal{J}}=0.
\end{eqnarray}

This completes the proof.
\end{proof}

% \begin{corollary}[Exact Error as a Schur Complement]
% \label{cor:qaca-schur-error}
% Under the assumptions of Theorem~\ref{thm:qaca-cur-equivalence},
% after suitable row and column permutations, write
% \[
% A=
% \begin{bmatrix}
% W_k & F_k\\
% G_k & D_k
% \end{bmatrix}.
% \]
% Then
% \[
% A-A_k
% =
% \begin{bmatrix}
% 0&0\\
% 0&D_k-G_kW_k^{-1}F_k
% \end{bmatrix}.
% \]
% Consequently,
% \[
% {
% \|A-A_k\|_F
% =
% \|D_k-G_kW_k^{-1}F_k\|_F,
% }
% \]
% and
% \[
% {
% \|A-A_k\|_2
% =
% \|D_k-G_kW_k^{-1}F_k\|_2.
% }
% \]
% \end{corollary}

% \begin{proof}
% By Theorem~\ref{thm:qaca-cur-equivalence},
% \[
% A_k=C_kW_k^{-1}R_k.
% \]
% After permuting the selected rows and columns to the leading
% positions,
% \[
% C_k=
% \begin{bmatrix}
% W_k\\G_k
% \end{bmatrix},
% \qquad
% R_k=
% \begin{bmatrix}
% W_k&F_k
% \end{bmatrix}.
% \]
% Therefore
% \[
% A_k
% =
% \begin{bmatrix}
% W_k\\G_k
% \end{bmatrix}
% W_k^{-1}
% \begin{bmatrix}
% W_k&F_k
% \end{bmatrix}
% =
% \begin{bmatrix}
% W_k&F_k\\
% G_k&G_kW_k^{-1}F_k
% \end{bmatrix}.
% \]
% Subtracting from $A$ gives
% \[
% A-A_k
% =
% \begin{bmatrix}
% 0&0\\
% 0&D_k-G_kW_k^{-1}F_k
% \end{bmatrix}.
% \]
% The stated norm identities follow immediately.
% \end{proof}

\begin{theorem}[Recursive Frobenius bound for $W_k^{-1}$]
\label{thm:frob-recursion}
Under the assumptions of Theorem~\ref{thm:qaca-cur-equivalence}, define
\[
 s_k=\norm{W_k^{-1}}_{\F}.
\]
Then
\[
\begin{aligned}
s_{k+1}
\le {}&
s_k
+\frac{1}{\abs{\alpha_{k+1}}}
\Big(
s_k^2\norm{p_{k+1}}_2\norm{q_{k+1}}_2
\\
&\hspace{25mm}
+s_k\norm{p_{k+1}}_2
+s_k\norm{q_{k+1}}_2
+1
\Big).
\end{aligned}
\]
In particular,
\[
s_1=\frac{1}{\abs{\alpha_1}}.
\]
\end{theorem}

\begin{proof}
By Theorem~\ref{thm:qaca-cur-equivalence} and the triangle inequality,
\[
\begin{aligned}
\norm{W_{k+1}^{-1}}_{\F}
\le {}&
\norm{W_k^{-1}}_{\F}
+
\abs{\alpha_{k+1}}^{-1}
\norm{W_k^{-1}p_{k+1}q_{k+1}W_k^{-1}}_{\F}
\\
&+
\abs{\alpha_{k+1}}^{-1}
\norm{W_k^{-1}p_{k+1}}_{\F}
\\
&+
\abs{\alpha_{k+1}}^{-1}
\norm{q_{k+1}W_k^{-1}}_{\F}
+
\abs{\alpha_{k+1}}^{-1}.
\end{aligned}
\]
Using
\[
\norm{XYZ}_{\F}\leq
\norm{X}_{\F}\norm{Y}_2\norm{Z}_2,
\]
and the corresponding two-factor inequalities gives
\[
\norm{W_k^{-1}p_{k+1}q_{k+1}W_k^{-1}}_{\F}
\leq
s_k^2\norm{p_{k+1}}_2\norm{q_{k+1}}_2,
\]
\[
\norm{W_k^{-1}p_{k+1}}_{\F}
\leq s_k\norm{p_{k+1}}_2,
\qquad
\norm{q_{k+1}W_k^{-1}}_{\F}
\leq s_k\norm{q_{k+1}}_2.
\]
Substitution proves the result.
\end{proof}

\begin{definition}[QACA inverse growth factor]
For $k\geq1$, define
\[
\gamma_{k+1}
=
\max\left\{
\norm{p_{k+1}}_2,\,
\norm{q_{k+1}}_2
\right\}.
\]
\end{definition}

\begin{corollary}[Inverse bound in terms of pivots and growth factors]
\label{cor:growth}
With $s_k=\norm{W_k^{-1}}_{\F}$ and $\gamma_{k+1}$ as above,

\[
s_{k+1}\le s_k+\frac{s_k^2\gamma_{k+1}^2+2s_k\gamma_{k+1}+1}{|\alpha_{k+1}|}
=s_k+\frac{(s_k\gamma_{k+1}+1)^2}{|\alpha_{k+1}|}.  
\]

%\[
%s_{k+1}
%\leq
%s_k+
%\frac{s_k\gamma_{k+1}^2+2\gamma_{k+1}+1}
%{\abs{\alpha_{k+1}}}.  
%\]
 
Consequently,
\[
s_k
\leq
s_1+
\sum_{\ell=1}^{k-1}
\frac{(s_\ell\gamma_{\ell+1}+1)^2}
{\abs{\alpha_{\ell+1}}},
\qquad
s_1=\frac1{\abs{\alpha_1}}.
\]
\end{corollary}

\begin{proof}
Since
\[
\norm{p_{k+1}}_2\leq\gamma_{k+1},
\qquad
\norm{q_{k+1}}_2\leq\gamma_{k+1},
\]
the estimate in Theorem~\ref{thm:frob-recursion} gives the first
inequality. Iteration gives the second.
\end{proof}

\begin{corollary}[Uniform pivot--growth bound]
\label{cor:uniform}
Suppose that for some $\alpha_{\min}>0$ and $\gamma\geq0$,
\[
\abs{\alpha_\ell}\geq\alpha_{\min},
\qquad
\gamma_\ell\leq\gamma
\]
for all relevant $\ell$. Then
\[
s_{k+1}
\leq
s_k+\frac{(\gamma\,s_k+1)^2}{\alpha_{\min}}.
\]
For $\gamma=0$ the recursion reduces to
\[
s_k\leq
s_1+\frac{k-1}{\alpha_{\min}}.
\]
Additionally, if we assume that $(\gamma s_k+1)^2\leq L (\gamma s_k+1),$ for some constant $L.$ Then,
\[
s_k
\leq
\left(s_1+\frac{1}{\gamma}\right)
\left(1+\frac{L\gamma}{\alpha_{min}}\right)^{k-1}
-\frac{1}{\gamma},
\qquad \gamma>0,
\]
% Therefore
% \[
% s_k
% \leq
% \left(s_1+\frac{2\gamma+1}{\gamma^2}\right)
% \left(1+\frac{\gamma^2}{\alpha_{\min}}\right)^{k-1}
% -\frac{2\gamma+1}{\gamma^2},
% \qquad \gamma>0,
% \]
% \[
% s_k
% \leq
% \left(s_1+\frac{2\gamma+1}{\gamma^2}\right)
% \left(1+\frac{\gamma^2}{\alpha_{\min}}\right)^{k-1}
% -\frac{2\gamma+1}{\gamma^2},
% \qquad \gamma>0,
% \]
where
\[
s_1\leq\frac1{\alpha_{\min}}.
\]

\end{corollary}

\begin{proof}
The first inequality follows immediately from Corollary~\ref{cor:growth}.
For $\gamma>0$, put
\[
c=\frac{L\gamma}{\alpha_{\min}},
\qquad
d=\frac{L}{\alpha_{\min}}.
\]
% \[
% c=\frac{\gamma^2}{\alpha_{\min}},
% \qquad
% d=\frac{2\gamma+1}{\alpha_{\min}}.
% \]
Then
\[
s_{k+1}\leq(1+c)s_k+d.
\]
Iterating this affine recursion gives
\[
s_k\leq(1+c)^{k-1}s_1+
d\frac{(1+c)^{k-1}-1}{c},
\]
which is the stated result.
The case $\gamma=0$ follows directly.
\end{proof}

%\begin{corollary}
%\label{cor:uniform}
%Suppose that for some $\alpha_{\min}>0$ and $\gamma\geq0$,
%\[
%\abs{\alpha_\ell}\geq\alpha_{\min},
%\qquad
%\gamma_\ell\leq\gamma
%\]
%for all relevant $\ell$. Then
%\[
%s_{k+1}
%\leq
%s_k+\frac{(\gamma\,s_k+1)^2}{\alpha_{\min}}.
%\]
%Therefore
%\[
%s_k
%\leq
%\left(s_1+\frac{2\gamma+1}{\gamma^2}\right)
%\left(1+\frac{\gamma^2}{\alpha_{\min}}\right)^{k-1}
%-\frac{2\gamma+1}{\gamma^2},
%\qquad \gamma>0,
%\]
%where
%\[
%s_1\leq\frac1{\alpha_{\min}}.
%\]
%For $\gamma=0$ the recursion reduces to
%\[
%s_k\leq
%s_1+\frac{k-1}{\alpha_{\min}}.
%\]
%\end{corollary}
%
%\begin{proof}
%The first inequality follows immediately from Corollary~\ref{cor:growth}.
%For $\gamma>0$, put
%\[
%c=\frac{\gamma^2}{\alpha_{\min}},
%\qquad
%d=\frac{2\gamma+1}{\alpha_{\min}}.
%\]
%Then
%\[
%s_{k+1}\leq(1+c)s_k+d.
%\]
%Iterating this affine recursion gives
%\[
%s_k\leq(1+c)^{k-1}s_1+
%d\frac{(1+c)^{k-1}-1}{c},
%\]
%which is the stated result.
%The case $\gamma=0$ follows directly.
%\end{proof}

\begin{theorem}[Cross-Conditioning and Error Bound of QACA]
\label{thm:qaca-cross-conditioning}
Let $A\in\mathbb{H}^{m\times n},$
and suppose that QACA selects $k$ rows and $k$ columns with index
sets
\[
I_k=\{i_1,i_2,\ldots,i_k\},
\qquad
J_k=\{j_1,j_2,\ldots,j_k\}.
\]
Define
\[
C_k=A(:,J_k)\in\mathbb{H}^{m\times k},\quad R_k=A(I_k,:)\in\mathbb{H}^{k\times n},
\]
and the selected intersection matrix
\[
W_k=A(I_k,J_k)\in\mathbb{H}^{k\times k}.
\]
Assume that $W_k$ is nonsingular. Let us  permute rows and columns so that the selected rows and
columns appear first
\[
A=
\begin{bmatrix}
W_k & F_k\\
G_k & D_k
\end{bmatrix}.
\]
Then, we have
\[
{
\|A-A_k\|_F
\leq
\|D_k\|_F+
\|G_k\|_2\|W_k^{-1}\|_2\|F_k\|_F,
}
\]
and
\[
{
\|A-A_k\|_2
\leq
\|D_k\|_2+
\|G_k\|_2\|W_k^{-1}\|_2\|F_k\|_2,
}
\]
where $A_k,$ is the approximation computed by the QACA. Thus, for fixed $F_k$ and $G_k$, the accuracy and stability of the
QACA approximation are directly controlled by
\[
\|W_k^{-1}\|_2
=
\frac{1}{\sigma_{\min}(W_k)}.
\]
\end{theorem}

\begin{proof}
By the QCUR equivalence theorem, the approximation generated after
$k$ successful QACA steps is
\[
A_k=C_kW_k^{-1}R_k.
\]
After permuting rows and columns so that the selected rows and
columns appear first, we have
\[
C_k=
\begin{bmatrix}
W_k\\
G_k
\end{bmatrix},
\qquad
R_k=
\begin{bmatrix}
W_k & F_k
\end{bmatrix}.
\]
Therefore,
\[
\begin{aligned}
A_k
&=
\begin{bmatrix}
W_k\\
G_k
\end{bmatrix}
W_k^{-1}
\begin{bmatrix}
W_k & F_k
\end{bmatrix}.
\end{aligned}
\]
Since quaternion matrix multiplication is associative, we may first
compute
\[
\begin{bmatrix}
W_k\\
G_k
\end{bmatrix}
W_k^{-1}
=
\begin{bmatrix}
I_k\\
G_kW_k^{-1}
\end{bmatrix}.
\]
Hence
\[
A_k
=
\begin{bmatrix}
I_k\\
G_kW_k^{-1}
\end{bmatrix}
\begin{bmatrix}
W_k&F_k
\end{bmatrix},
\]
which gives
\[
A_k=
\begin{bmatrix}
W_k&F_k\\
G_k&G_kW_k^{-1}F_k
\end{bmatrix}.
\]
Subtracting this expression from
\[
A=
\begin{bmatrix}
W_k&F_k\\
G_k&D_k
\end{bmatrix}
\]
gives
\[
A-A_k
=
\begin{bmatrix}
0&0\\
0&D_k-G_kW_k^{-1}F_k
\end{bmatrix}.
\]
Therefore,
\[
\|A-A_k\|_F
=
\|D_k-G_kW_k^{-1}F_k\|_F
\]
and
\[
\|A-A_k\|_2
=
\|D_k-G_kW_k^{-1}F_k\|_2.
\]
Using the triangle inequality,
\[
\begin{aligned}
\|D_k-G_kW_k^{-1}F_k\|_F
&\leq
\|D_k\|_F+
\|G_kW_k^{-1}F_k\|_F.
\end{aligned}
\]
By the submultiplicativity of the Frobenius and spectral norms,
\[
\|G_kW_k^{-1}F_k\|_F
\leq
\|G_k\|_2
\|W_k^{-1}\|_2
\|F_k\|_F.
\]
Hence
\[
\|A-A_k\|_F
\leq
\|D_k\|_F+
\|G_k\|_2
\|W_k^{-1}\|_2
\|F_k\|_F.
\]
Similarly,
\[
\begin{aligned}
\|A-A_k\|_2
&=
\|D_k-G_kW_k^{-1}F_k\|_2\\
&\leq
\|D_k\|_2+
\|G_k\|_2
\|W_k^{-1}\|_2
\|F_k\|_2.
\end{aligned}
\]
Finally, since $W_k$ is nonsingular,
\[
\|W_k^{-1}\|_2
=
\frac{1}{\sigma_{\min}(W_k)}.
\]
This proves the result. The last fact shows that the accuracy and stability of the
QACA approximation are directly controlled by the spectral norm of the pivot matrix $
\|W_k^{-1}\|_2
=
\frac{1}{\sigma_{\min}(W_k)},$
because $F_k$ and $G_k$, are fixed.
% After simultaneous permutation of the selected rows and columns,
% $A$ can be written as
% \[
% A=
% \begin{bmatrix}
% W_k & F_k\\
% G_k & D_k
% \end{bmatrix},
% \]
% and
% \[
% A_k=
% \begin{bmatrix}
% W_k & F_k\\
% G_k & G_kW_k^{-1}F_k
% \end{bmatrix}.
% \]
% Consequently,
% \[
% A-A_k=
% \begin{bmatrix}
% 0&0\\
% 0&D_k-G_kW_k^{-1}F_k
% \end{bmatrix}.
% \]
% In particular,
% \[
% {
% \|A-A_k\|_F
% =
% \|D_k-G_kW_k^{-1}F_k\|_F
% }
% \]
% and
% \[
% {
% \|A-A_k\|_2
% =
% \|D_k-G_kW_k^{-1}F_k\|_2.
% }
% \]
% Moreover,
% \[
% {
% \|A-A_k\|_F
% \leq
% \|D_k\|_F+
% \|G_k\|_2\|W_k^{-1}\|_2\|F_k\|_F,
% }
% \]
% and
% \[
% {
% \|A-A_k\|_2
% \leq
% \|D_k\|_2+
% \|G_k\|_2\|W_k^{-1}\|_2\|F_k\|_2.
% }
% \]
% Thus, for fixed $F_k$ and $G_k$, the accuracy and stability of the
% QACA approximation are directly controlled by
% \[
% \|W_k^{-1}\|_2
% =
% \frac{1}{\sigma_{\min}(W_k)}.
% \]
\end{proof}

Note that in cross approximation, it is not necessary for the number of selected rows and column indices to be the same. If the number of row indices differs from the number of column indices, we replace Moore-Penrose inverse instead of inverse for computing the middle matrix in the CUR approximation. However, the QACA always provides approximations with exactly the same number of row and column indices.

\begin{corollary}[Residual bound using pivots and growth factors]
\label{cor:residual-growth}
Under the assumptions of Corollary~\ref{cor:growth},
\[
\norm{R_k}_{\F}
\leq
\norm{D_k}_{\F}
+
\norm{G_k}_2
\left[
\frac1{\abs{\alpha_1}}
+
+
\sum_{\ell=1}^{k-1}
\frac{(s_\ell\gamma_{\ell+1}+1)^2}
{\abs{\alpha_{\ell+1}}}
\right]
\norm{F_k}_{\F}.
\]
Thus the residual is controlled by the pivots together with the
interpolation/elimination growth factors.
\end{corollary}

\begin{proof}
Substitute the recursive estimate of Corollary~\ref{cor:growth}
into Theorem~\ref{thm:qaca-cross-conditioning}.
\end{proof}

\section{Computational Complexity of QACA}\label{sec:complex}
The native quaternion ACA algorithm (Algorithm 1 in the paper) constructs a rank-$k$ approximation $A \approx UV^T$, where $U \in \mathbb{H}^{m \times k}$ and $V \in \mathbb{H}^{n \times k}$. The algorithm proceeds iteratively, adding one rank at a time. At each iteration $k$ (where $k \leq r$, the final rank), the algorithm performs the following key operations:

\begin{enumerate}
    \item \textbf{Residual Row Computation:} 
    \[
    r = A(i,:) - U(i,1:k-1) V(:,1:k-1)^T
    \]
    The product $U(i,1:k-1) \in \mathbb{H}^{1 \times (k-1)}$ and $V(:,1:k-1)^T \in \mathbb{H}^{(k-1) \times n}$ costs $\mathcal{O}(k n)$ quaternion multiplications.

    \item \textbf{Residual Column Computation:}
    \[
    c = A(:,j) - U(:,1:k-1) V(j,1:k-1)^T
    \]
    The product $U(:,1:k-1) \in \mathbb{H}^{m \times (k-1)}$ and $V(j,1:k-1)^T \in \mathbb{H}^{(k-1) \times 1}$ costs $\mathcal{O}(m k)$ quaternion multiplications.

    \item \textbf{Pivot Selection:} Finding the maximum entry in vectors of length $n$ and $m$ costs $\mathcal{O}(n)$ and $\mathcal{O}(m)$ operations, respectively.
\end{enumerate}
Thus, the dominant cost per iteration is $\mathcal{O}(k(m+n))$. Since $k \leq r$, the per-iteration cost is bounded by $\mathcal{O}(r(m+n))$. Summing over $r$ iterations, the total computational complexity is:
\[
\sum_{k=1}^{r} \mathcal{O}(k(m+n)) = \mathcal{O}\left((m+n) \sum_{k=1}^{r} k\right) = \mathcal{O}(r^2(m+n))
\]
The key point is the linear dependence on $m$ and $n$, which makes the algorithm highly efficient for large matrices. The computational complexity of the QSVD is typically $\mathcal{O}(mn \min(m,n))$ for full decomposition, or $\mathcal{O}(mnr)$ for computing only the first $r$ singular values and vectors. This is significantly higher than the $\mathcal{O}(r(m+n))$ complexity of QACA, especially for large matrices where $r \ll \min(m,n)$.

\begin{table}[h]
\caption{Computational Complexity Comparison}
\centering
\begin{tabular}{|c|c|c|}
\hline
\textbf{Algorithm} & \textbf{Complexity} & \textbf{Key Features} \\
\hline
QSVD & $\mathcal{O}(mn\min(m,n))$ & Optimal approximation, full decomposition \\
QACA & $\mathcal{O}(r^2(m+n))$ & Near-optimal, data-efficient, memory-efficient \\
\hline
\end{tabular}
\end{table}

So, the QACA is better than QSVD from the following aspects:

1. \textbf{Scalability:} The linear complexity in $m$ and $n$ makes QACA suitable for large-scale problems where QSVD is impractical.

2. \textbf{Memory Efficiency:} QACA only accesses selected rows and columns of the matrix, avoiding the need to store the full matrix explicitly.

3. \textbf{Preservation of Structure:} Operating natively in the quaternion algebra preserves the algebraic structure and natural correlations between quaternion components, unlike real embedding approaches that suffer from dimensionality blow-up and information loss.

\section{Numerical Experiments}\label{sec:experiments}
In this section we examine the proposed method for low-rank approximation of quaternion matrices. We consider medium-sized and large-scale quaternion matrices in our experiments. All experiments are performed in MATLAB using the QTFM toolbox for quaternion arithmetic. The QACA algorithm is compared against the standard SVD algorithm.

\begin{example}
 We first evaluate the performance of QACA on randomly generated quaternion matrices of small dimensions. For each test, we generate a quaternion matrix $A=U_{\text{true}} V_{\text{true}}^T \in \mathbb{H}^{m \times n}$, where $U_{\text{true}} \in \mathbb{H}^{m \times r_{\text{true}}}$ and $V_{\text{true}} \in \mathbb{H}^{n \times r_{\text{true}}}$ have random quaternion entries\footnote{We say a quaternion $
 a+ b\mathbf{i} + c\mathbf{j} + d\mathbf{k},
$ is a random quaternion if
 $a, b, c, d \sim \mathcal{N}(0,1)$ are independent standard normal random variables}.  The experimental setup is as follows:
\begin{itemize}
    \item Matrix dimensions: $m=n = 50,100,150,\ldots,400$. 
    \item True rank: $r_{\text{true}}=30$.
    \item Tolerance: $\varepsilon =10^{-10}$.
    \item Maximum rank: $r_{\max} = 32$.
\end{itemize}
For each configuration, we perform 50 independent trials and report the relative error and computing time in Table \ref{tab:results} and Figures \ref{fig:error_sub}-\ref{fig:time_sub}. The numerical results clearly indicate that the proposed QACA delivers approximations of comparable accuracy to existing methods while requiring substantially less computational effort. Moreover, the performance advantage of QACA becomes increasingly significant as the matrix dimensions grow. To illustrate, for a quaternion matrix of dimension $400 \times 400$, QACA achieves a speedup of approximately $78\times$ over the conventional quaternion SVD.

\begin{figure}[htbp]\label{fig_1}
    \centering
    \begin{subfigure}[b]{0.48\textwidth}
        \centering
        \includegraphics[width=\textwidth]{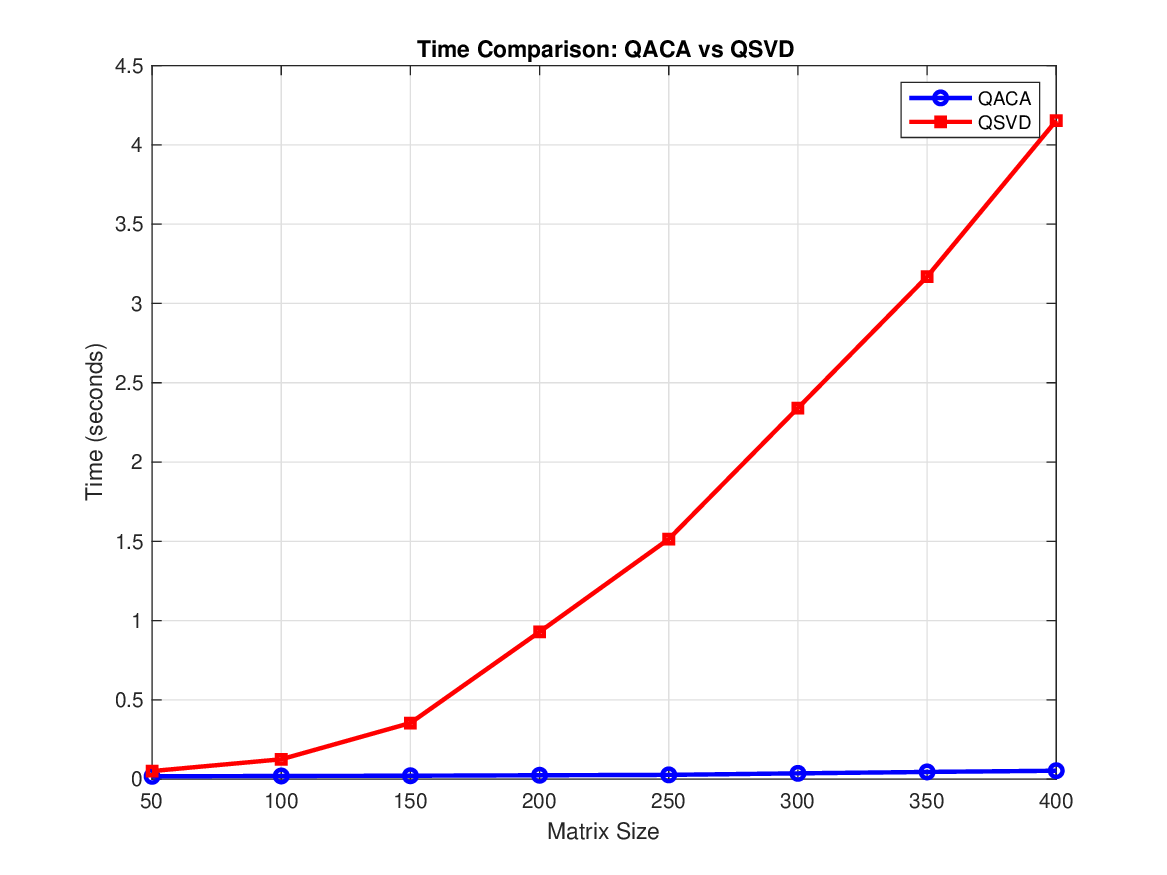}
        \caption{Time comparison}
        \label{fig:time_sub}
    \end{subfigure}
    \hfill
    \begin{subfigure}[b]{0.48\textwidth}
        \centering
        \includegraphics[width=\textwidth]{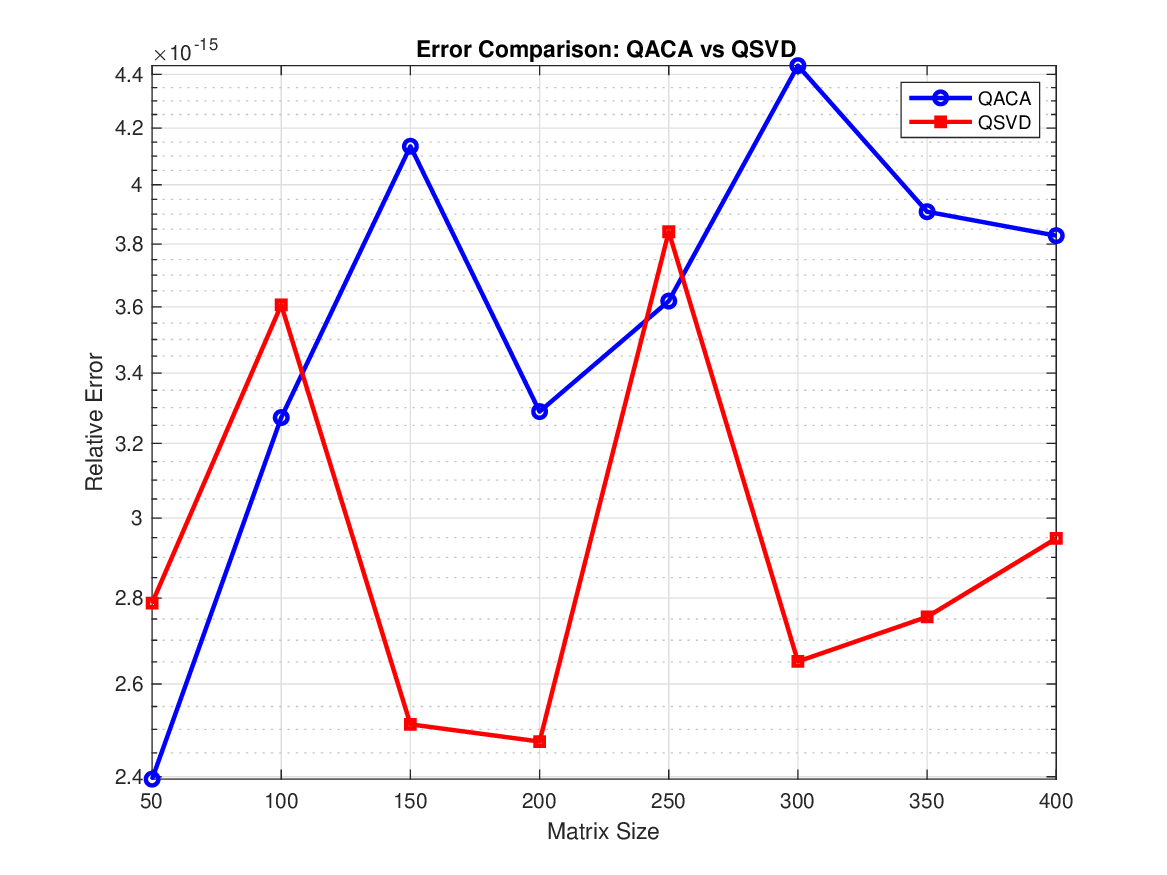}
        \caption{Error comparison}
        \label{fig:error_sub}
    \end{subfigure}
    \caption{Performance comparison between Adaptive Cross Approximation (ACA) and Singular Value Decomposition (SVD). (a) Computation time as a function of matrix size, showing ACA's superior computational efficiency across all tested matrix sizes. (b) Relative errors on a semilogarithmic scale, demonstrating that both methods achieve comparable accuracy at the machine precision level ($\sim 10^{-15}$).}
    \label{fig:combined_comparison}
\end{figure}

We also examined the performance of the QACA on noisy data. To this end, we constructed the following noisy quaternion matrix:
\[
Y = U_{\text{true}} V_{\text{true}}^T + \delta X,
\]
where \(U_{\text{true}}\in\mathbb{H}^{100\times 20}\), \(V_{\text{true}}\in\mathbb{H}^{100\times 20}\), and \(X\in\mathbb{H}^{100\times 100}\) are random matrices, and \(\delta = 0.01\). The singular values of this matrix, along with the corresponding tolerance threshold \(\delta \|Y\|_F\), are shown in Figure~\ref{fig:noisy}. We applied the QACA to the mentioned matrix using tolerance $1e-5$ and the maximum rank as $r=30$. The error history of the QACA applied to the aforementioned matrix, as new rank-1 terms are successively added, is shown in Figure~\ref{fig:noisy_2} (right). The original matrix and the reconstructed matrices are visualized in Figure~\ref{fig:noisy_2} (left) and (middle), respectively. The QACA identified the rank $20$. Additionally, we applied Algorithm \ref{alg:qaca_auto}, which does not require the maximum rank parameter. Using only the tolerance $\epsilon=1e-5$, the algorithm again detected rank 20. This clearly demonstrates the efficiency of the proposed adaptive cross algorithm for determining the rank.

\begin{figure}[htbp]
    \centering    \includegraphics[width=0.8\textwidth]{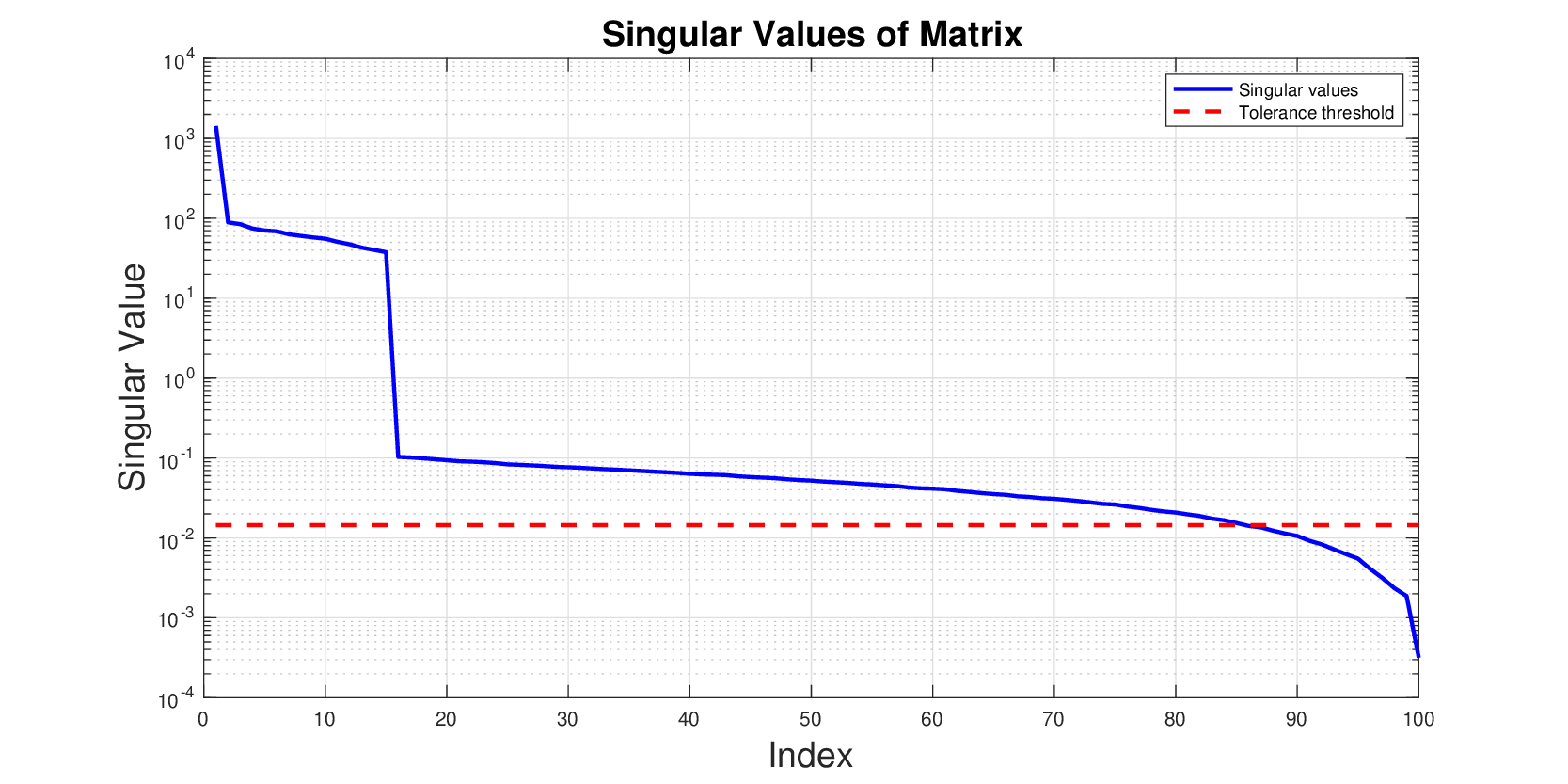}
    \caption{Singular values of the noisy quaternion matrix $Y = U_{\text{true}} V_{\text{true}}^T + \delta X$ with $\delta = 0.01$. The horizontal dashed line indicates the tolerance threshold $\delta \|Y\|_F$, used to distinguish significant singular values from noise.}
    \label{fig:noisy}
\end{figure}

\begin{figure}[htbp]
    \centering
\includegraphics[width=1.1\textwidth]{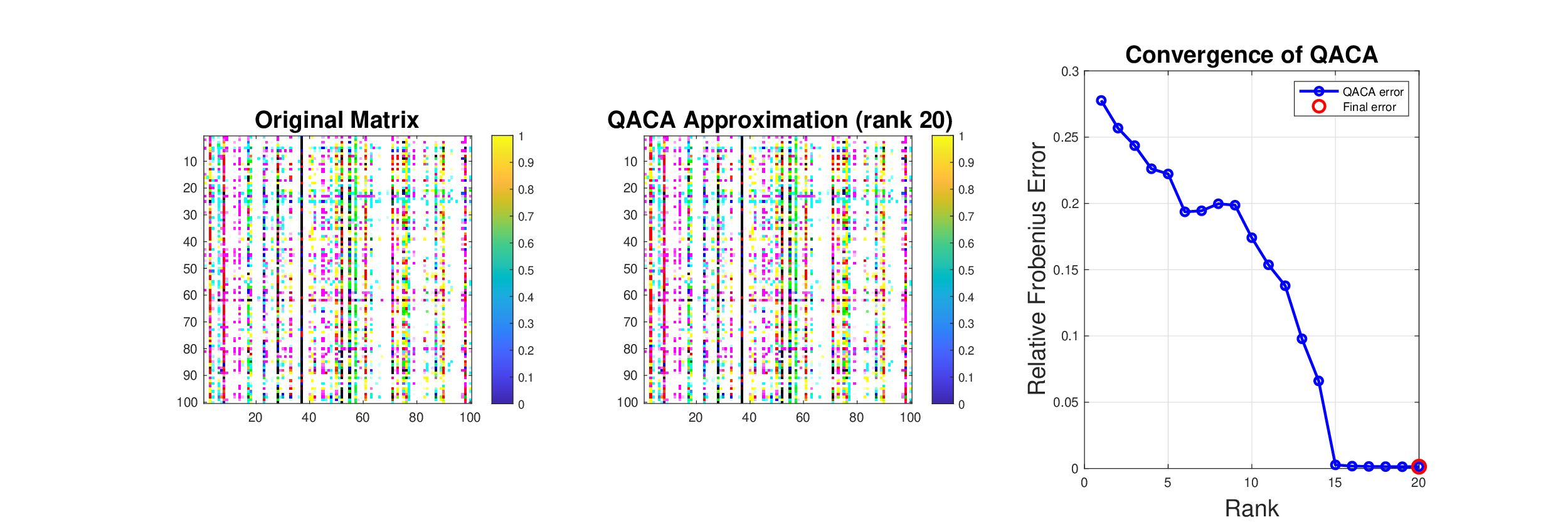}
    \caption{Performance of the QACA on the noisy quaternion matrix. (Right) Error history as rank-1 terms are successively added, showing the decay of the approximation error. (Left) The original noisy matrix. (Middle) The reconstructed matrix obtained by the QACA.}
    \label{fig:noisy_2}
\end{figure}

\begin{table}[H]
    \centering
    \caption{Experimental results for QACA and QSVD methods}
    \label{tab:results}
    \begin{tabular}{|c|r|r|r|r|}
        \hline
        \textbf{Size} & \multicolumn{2}{c|}{\textbf{Time (s)}} & \multicolumn{2}{c|}{\textbf{Error}} \\
        \cline{2-5}
        & \textbf{QACA} & \textbf{QSVD} & \textbf{QACA} & \textbf{QSVD} \\
        \hline
        50 & 0.017895 & 0.051062 & $2.395\times10^{-15}$ & $2.788\times10^{-15}$ \\
        100 & 0.020097 & 0.125350 & $3.272\times10^{-15}$ & $3.606\times10^{-15}$ \\
        150 & 0.021546 & 0.354177 & $4.135\times10^{-15}$ & $2.511\times10^{-15}$ \\
        200 & 0.024696 & 0.929008 & $3.289\times10^{-15}$ & $2.474\times10^{-15}$ \\
        250 & 0.027009 & 1.514313 & $3.618\times10^{-15}$ & $3.841\times10^{-15}$ \\
        300 & 0.036696 & 2.340746 & $4.433\times10^{-15}$ & $2.651\times10^{-15}$ \\
        350 & 0.045579 & 3.168791 & $3.908\times10^{-15}$ & $2.755\times10^{-15}$ \\
        400 & 0.053179 & 4.153262 & $3.828\times10^{-15}$ & $2.948\times10^{-15}$ \\
        \hline
    \end{tabular}
\end{table}
  
\end{example}

\begin{example}(large-scale random matrices)
 We next evaluate the scalability of QACA on large-scale quaternion matrices. These experiments are designed to assess the algorithm's performance in realistic applications where matrix dimensions are substantial and full matrix storage is infeasible.

The experimental configuration for large-scale tests includes:
\begin{itemize}
    \item Matrix dimensions: $m = n = 1000, 2000, 3000, 4000$
    \item True rank: $r_{\text{true}} = 50$
    % \item Tolerance: $\varepsilon = 10^{-4}, 10^{-6}$
    \item Maximum rank: $r_{\max} = 100$
\end{itemize}

\begin{table}[H]
    \centering
    \caption{Experimental results for QACA and QSVD methods for $r_{\text{true}} = 50$.}
    \label{tab:results_2}
    \begin{tabular}{|c|r|r|r|r|}
        \hline
        \textbf{Size} & \multicolumn{2}{c|}{\textbf{Time (s)}} & \multicolumn{2}{c|}{\textbf{Error}} \\
        \cline{2-5}
        & \textbf{QACA} & \textbf{QSVD} & \textbf{QACA} & \textbf{QSVD} \\
        \hline
        1000 &  0.508352 &  27.423186 & $1.397\times 10^{-14}$ & $3.331\times 10^{-15}$ \\
        2000 & 2.948609 & 154.414234 & $1.063\times 10^{-14}$ & $3.109\times 10^{-15}$ \\
        3000 & 10.077973 & 650.964245 & $1.523\times{10}^{-14}$ & $4.3456\times{10}^{-15}$ \\
        4000 & 22.415436 & 970.235418 & $2.511\times 10^{-14}$ & $2.5116\times 10^{-15}$ \\
        \hline
    \end{tabular}
\end{table}
The results are reported in Table \ref{tab:results_2}. The experiments demonstrate that the QACA achieves comparable accuracy to the QSVD while offering significantly greater scalability. We also repeated the experiments with true ranks $r_{\text{true}} = 10,\,30$. These results are reported in tables \ref{tab:results_3} and \ref{tab:results_4}.

\begin{table}[H]
    \centering
    \caption{Experimental results for QACA and QSVD methods for $r_{\text{true}} = 10$.}
    \label{tab:results_3}
    \begin{tabular}{|c|r|r|r|r|}
        \hline
        \textbf{Size} & \multicolumn{2}{c|}{\textbf{Time (s)}} & \multicolumn{2}{c|}{\textbf{Error}} \\
        \cline{2-5}
        & \textbf{QACA} & \textbf{QSVD} & \textbf{QACA} & \textbf{QSVD} \\
        \hline
        1000 &  0.455614 &  13.114 & $2.062\times 10^{-15}$ & $3.241\times 10^{-15}$ \\
        2000 & 3.059793 & 92.851 & $2.533\times 10^{-15}$ & $3.512\times 10^{-15}$ \\
        3000 &  9.300740 & 407.150 & $1.535\times 10^{-15}$ & $3.987\times 10^{-15}$ \\
        4000 & 21.684523 & 752.403 & $1.221\times 10^{-15}$ & $4.102\times 10^{-15}$ \\
        \hline
    \end{tabular}
\end{table}

\begin{table}[H]
    \centering
    \caption{Experimental results for QACA and QSVD methods for $r_{\text{true}} = 30$.}
    \label{tab:results_4}
    \begin{tabular}{|c|r|r|r|r|}
        \hline
        \textbf{Size} & \multicolumn{2}{c|}{\textbf{Time (s)}} & \multicolumn{2}{c|}{\textbf{Error}} \\
        \cline{2-5}
        & \textbf{QACA} & \textbf{QSVD} & \textbf{QACA} & \textbf{QSVD} \\
        \hline
        1000 & 0.451124 & 19.431 & $1.183\times 10^{-14}$ & $3.411\times 10^{-15}$ \\
        2000 & 3.103619 & 127.598 & $4.563\times 10^{-15}$ & $3.892\times 10^{-15}$ \\
        3000 & 9.195136 & 552.765 & $1.829\times 10^{-14}$ & $4.211\times 10^{-15}$ \\
        4000 & 22.231566 & 1010.541 & $3.432\times 10^{-15}$ & $4.523\times 10^{-15}$ \\
        \hline
    \end{tabular}
\end{table}

The new experiments reveal a clear and significant performance advantage of the QACA over the QSVD.

% For these large-scale tests, we use function handles to define the matrix entries implicitly, avoiding the need to store the full matrix:
% \[
% A(i,j) = \sum_{t=1}^{r_{\text{true}}} u_{it} \overline{v_{jt}},
% \]
% where $u_{it}, v_{jt}$ are random quaternion entries. This construction ensures the matrix has exact low-rank structure while allowing on-the-fly entry evaluation.
   
\end{example}

\begin{example}
(Image compression) In this example, we apply the proposed algorithm for compressing color images. Consider three images shown in Figure~\ref{fig:img}. The size of all images is $256 \times 256 \times 3$. These images are first converted to pure quaternions, where each pixel's RGB values are encoded as the imaginary parts of a quaternion. The QSVD and QACA algorithms are then applied to them for ranks 10, 30, and 50, producing low-rank approximations of the original quaternion matrices. In this example we considered the QACA for selecting the column and row indices but the middle matrix by $U=C^\dagger AR^\dagger$,  where $\dagger$, denote the Moore-Penrose inverse operation. It was established in Theorem \ref{thm:qaca-cur-equivalence} that the QACA algorithm computes a CUR approximation whose middle matrix is $U=A(I_K,J_K)^\dagger$, with $I_K$ and $J_K$ denoting the selected indices. As reported in \cite{ahmadi2024adaptive}, although images and videos possess low-rank structures, they exhibit no small tail energy, and the choice $U=A(I_K,J_K)^\dagger$ leads to unstable approximations, whereas $U=C^\dagger AR^\dagger$ yields substantially more stable results. In this example, our strategy was to select the indices via QACA while improving the approximation through the use of better middle matrices.
The running times of QACA and QSVD are reported in Table~\ref{tab:runtime_comparison}. Additionally, the compressed images obtained by both methods are visualized in Figure~\ref{fig:compre} for qualitative comparison. As can be observed, both algorithms produce visually similar compressed images, especially for higher ranks.

Our main observation is that when we select a relatively large rank, the output of QACA becomes very close to that of QSVD while achieving a significant speed-up. As shown in Table~\ref{tab:runtime_comparison}, QACA achieves more than a $4\times$ speed-up in some cases, making it particularly attractive for real-time or large-scale image processing applications. This speed advantage stems from the fact that QACA is a column-pivoting, adaptive cross-approximation method that avoids the full decomposition, whereas QSVD requires computing the complete singular value decomposition, which is computationally more expensive.

It should be highlighted that the images considered in this example are not very large, with dimensions of only $256 \times 256$ pixels. When dealing with larger images, such as high-resolution photographs or medical images, the difference in running time between the two methods becomes even more pronounced. In such scenarios, the computational savings offered by QACA can be substantial, potentially reducing processing times from hours to minutes or from minutes to seconds.

\begin{figure}[htbp]
    \centering
\includegraphics[width=0.6\textwidth]{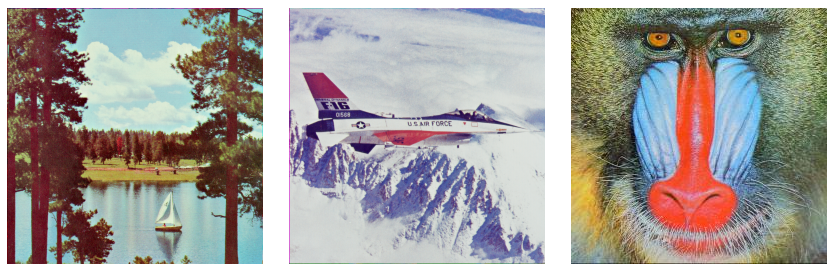}
    \caption{The sample images are used in our simulations.}
    \label{fig:img}
\end{figure}

\begin{table}[htbp]
\centering
\caption{Running Time Comparison of QACA and QSVD (in seconds)}
\label{tab:runtime_comparison}
\begin{tabular}{lrrrr}
\toprule
\multirow{2}{*}{Image} & \multirow{2}{*}{Rank} & \multicolumn{2}{c}{Time (s)} & \multirow{2}{*}{Speedup} \\
\cmidrule(lr){3-4}
 & & QACA & QSVD & \\
\midrule
\multirow{3}{*}{Baboon} & 10 & 0.386 & 1.378 & 3.57$\times$ \\
 & 30 & 0.458 & 1.382 & 3.02$\times$ \\
 & 50 & 0.645 & 1.394 & 2.16$\times$ \\
\multirow{3}{*}{Sailboat} & 10 & 0.322 & 1.402 & 4.35$\times$ \\
 & 30 & 0.462 & 1.502 & 3.25$\times$ \\
 & 50 & 0.645 & 1.472 & 2.28$\times$ \\
\multirow{3}{*}{Airplane} & 10 & 0.316 & 1.399 & 4.42$\times$ \\
 & 30 & 0.458 & 1.418 & 3.10$\times$ \\
 & 50 & 0.632 & 1.443 & 2.28$\times$ \\
\bottomrule
\end{tabular}
\end{table}

\begin{figure}[htbp]
    \centering
\includegraphics[width=1\textwidth]{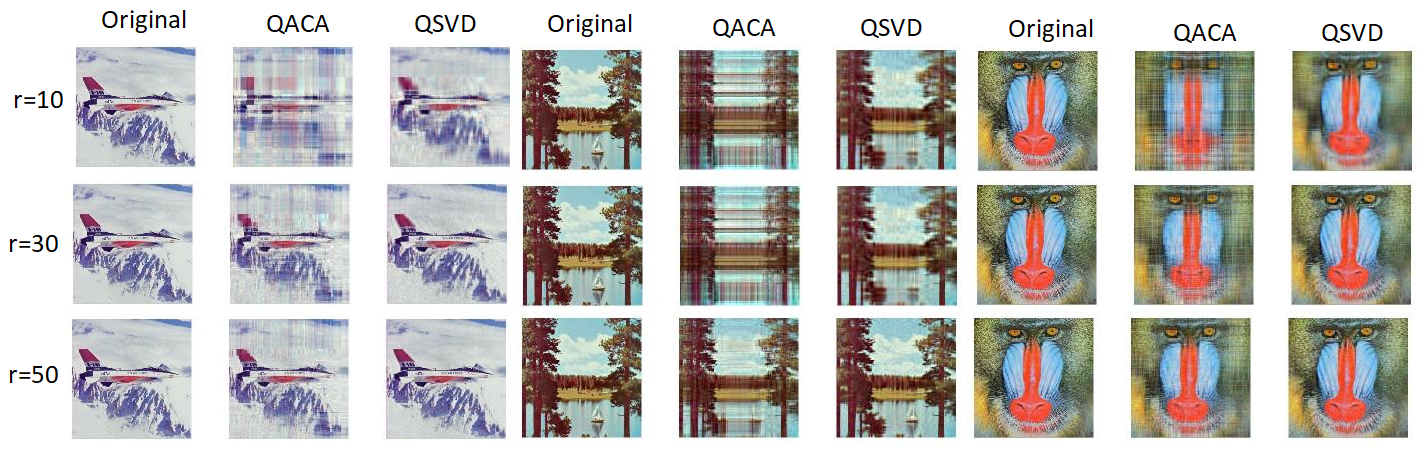}
    \caption{The samples images used in our simulations.}
    \label{fig:compre}
\end{figure}

\end{example}

\section{Conclusion}\label{sec:conclusion}
This paper has presented a comprehensive framework for adaptive cross approximation operating natively in the quaternion algebra $\mathbb{H}$. We have developed a fully quaternionic ACA algorithm that preserves algebraic structure throughout the computation, avoiding the information loss and computational overhead associated with real embedding approaches. The theoretical analysis establishes convergence guarantees, error bounds, and complexity estimates, proving that native quaternion ACA achieves quasi-optimal approximation properties comparable to the QSVD. The algorithm has complexity $\mathcal{O}(r^2(m+n))$, making it suitable for large-scale problems where full QSVD is impractical.

Our MATLAB implementation, using the QTFM toolbox, demonstrates the practical effectiveness of the approach, and numerical experiments show significant acceleration of the proposed adaptive cross approximation algorithm compared to the classical SVD.

The native quaternion ACA represents a significant advancement in quaternion matrix computations. By operating directly in the quaternion algebra, we achieve both theoretical elegance and practical efficiency, opening new possibilities for large-scale quaternionic computations in color image processing, signal processing, robotics, and beyond.

Future work in this direction includes developing theoretically grounded pivot selection criteria that account for the quaternionic structure, such as leveraging the relationship between quaternion matrices and their real or complex counterparts. Additionally, the integration of these pivoting strategies with rank estimation techniques could lead to fully adaptive algorithms that determine the rank automatically based on the data, further enhancing the practical utility of quaternion ACA in large-scale applications. Extending the QACA to split quaternion \cite{ahmadi2026iterative} is under investigation.

\bibliographystyle{plain}
\bibliography{ref}

\appendix

\section{Implementation Details}

\subsection{MATLAB Implementation Using QTFM Toolbox}

We implement the native quaternion ACA using the QTFM (Quaternion Toolbox for Matlab) \cite{qtfm}. The key implementation are:

\subsubsection{1. Quaternion Array Initialization}

\begin{lstlisting}[caption=Initializing empty quaternion arrays]
% Initialize quaternion factors with maximum rank
U = quaternion(zeros(m,maxrank), zeros(m,maxrank), zeros(m,maxrank), zeros(m,maxrank));
V = quaternion(zeros(n,maxrank), zeros(n,maxrank), zeros(n,maxrank), zeros(n,maxrank));
\end{lstlisting}

\subsubsection{2. Quaternion Arithmetic in Loops}

\begin{lstlisting}[caption=Computing residual row and column with quaternion arithmetic]
% Compute residual row (1 x n)
r = A(i,:);
if k > 1
    % V(:,1:k-1)' is (k-1) x n, U(i,1:k-1) is 1 x (k-1)
    % Product is 1 x n
    r = r - U(i,1:k-1) * V(:,1:k-1).';
end

% Compute residual column (m x 1)
c = A(:,j);
if k > 1
    % V(j,1:k-1)' is 1 x (k-1)
    % Product U(:,1:k-1) * V(j,1:k-1)' is m x 1
    c = c - U(:,1:k-1) * V(j,1:k-1).';
end
\end{lstlisting}

\subsubsection{3. Quaternion Normalization}

\begin{lstlisting}[caption=Quaternion normalization using inverse]
% Get pivot value (quaternion scalar)
alpha = r(j);

% Check for near-zero pivot
if abs(alpha) < 1e-15
    break;
end

% Normalize using quaternion inverse
V(:,k) = (1/alpha) * r.';  % r' is n x 1, divided by alpha
\end{lstlisting}

\subsubsection{4. Pivot Selection Strategy}

\begin{lstlisting}[caption=Adaptive pivot selection with row/column tracking]
selected_rows = false(m,1);
selected_cols = false(n,1);

% Initial row - find first non-zero row
i = 1;
for row_idx = 1:m
    if norm(A(row_idx,:)) > 0
        i = row_idx;
        break;
    end
end
selected_rows(i) = true;

for k = 1:maxrank
    % Select pivot column from unselected columns
    unselected_cols = find(~selected_cols);
    if isempty(unselected_cols)
        break;
    end
    
    [pivot_val, idx] = max(abs(r(unselected_cols)));
    j = unselected_cols(idx);
    
    if pivot_val <= tol * Anorm
        break
    end
    
    % Select next row from unselected rows
    unselected_rows = find(~selected_rows);
    if isempty(unselected_rows)
        break;
    end
    
    [~, idx] = max(abs(c(unselected_rows)));
    i = unselected_rows(idx);
    selected_rows(i) = true;
    selected_cols(j) = true;
end
\end{lstlisting}

\subsubsection{5. Error Computation}

\begin{lstlisting}[caption=Frobenius norm error computation]
% Compute approximation error using Frobenius norm
if r > 0
    Aapprox = U*V.';  % Now correctly m x n
    info.rank = r;
    info.error = norm(A-Aapprox,'fro') / Anorm;
else
    info.rank = 0;
    info.error = 1;
end
\end{lstlisting}

\subsection{Complete Implementation}

\begin{lstlisting}
 function [U,V,info] = QACA(A,maxrank,tol)

[m,n] = size(A);

U = quaternion(zeros(m,maxrank),zeros(m,maxrank),zeros(m,maxrank),zeros(m,maxrank));
V = quaternion(zeros(n,maxrank),zeros(n,maxrank),zeros(n,maxrank),zeros(n,maxrank));

Anorm = norm(A,'fro');

if Anorm == 0
    U = [];
    V = [];
    info.rank = 0;
    info.error = 0;
    return
end

selected_rows = false(m,1);
selected_cols = false(n,1);

% Initial row - find first non-zero row
i = 1;
for row_idx = 1:m
    if norm(A(row_idx,:)) > 0
        i = row_idx;
        break;
    end
end

if i > m
    U = [];
    V = [];
    info.rank = 0;
    info.error = 0;
    return;
end

selected_rows(i) = true;

r = 0;   % actual rank reached

for k = 1:maxrank

    %% Compute residual row
    r_row = A(i,:);

    if k > 1
        r_row = r_row - U(i,1:k-1) * V(:,1:k-1).';
    end

    %% Select pivot column from unselected columns
    unselected_cols = find(~selected_cols);
    if isempty(unselected_cols)
        break;
    end

    [pivot_val, idx] = max(abs(r_row(unselected_cols)));
    j = unselected_cols(idx);

    if pivot_val <= tol * Anorm
        %% ---- Random sampling verification step (from algorithm) ----
        % Build current approximation from first k-1 terms
        if k > 1
            Atilde = U(:,1:k-1) * V(:,1:k-1).';
        else
            Atilde = quaternion(zeros(m,n),zeros(m,n),zeros(m,n),zeros(m,n));
        end

        % Sample m+n random (row, col) indices
        numSamples = m + n;
        is = randi(m, numSamples, 1);
        js = randi(n, numSamples, 1);

        % Compute errors on sampled entries
        errS_sq = 0;
        nrmS_sq = 0;
        maxAbsDiff = 0;
        best_i = is(1);
        best_j = js(1);

        for s = 1:numSamples
            a_val   = A(is(s), js(s));
            at_val  = Atilde(is(s), js(s));
            diff    = abs(a_val - at_val);

            errS_sq = errS_sq + diff^2;
            nrmS_sq = nrmS_sq + abs(a_val)^2;

            if diff > maxAbsDiff
                maxAbsDiff = diff;
                best_i = is(s);
                best_j = js(s);
            end
        end

        errS = sqrt(errS_sq);
        nrmS = sqrt(nrmS_sq);

        if errS <= tol * nrmS
            % Accept current rank as final
            r = k - 1;
            break;
        else
            % Sampling suggests matrix is not yet resolved.
            % Continue with the sampled max entry as the new pivot.
            i = best_i;
            j = best_j;

            % Make sure these are not already selected (optional safety)
            selected_rows(i) = true;
            selected_cols(j) = true;

            % Fall through to compute the rank-1 update using this pivot
        end
        %% -----------------------------------------------------------
    end

    alpha = r_row(j);

    if abs(alpha) < 1e-15
        break;
    end

    %% Compute residual column
    c = A(:,j);

    if k > 1
        c = c - U(:,1:k-1) * V(j,1:k-1).';
    end

    %% New rank-one term
    U(:,k) = c;
    V(:,k) = (1/alpha)*r_row.';

    %% Select next row from unselected rows
    unselected_rows = find(~selected_rows);
    if isempty(unselected_rows)
        r = k;
        break;
    end

    [~, idx] = max(abs(c(unselected_rows)));
    i = unselected_rows(idx);
    selected_rows(i) = true;
    selected_cols(j) = true;

    r = k;   % update rank tracker
end

if r > 0
    U = U(:,1:r);
    V = V(:,1:r);
else
    U = [];
    V = [];
end

r = k - 1;

%% Information
if r > 0
    Aapprox = U*V.';
    info.rank = r;
    info.error = norm(A-Aapprox,'fro') / Anorm;
else
    info.rank = 0;
    info.error = 1;
end

end   
\end{lstlisting}

% The implementation includes several important features:

% \begin{enumerate}
%     \item \textbf{Memory Efficiency:} Pre-allocates quaternion arrays with maximum rank to avoid dynamic resizing.
%     \item \textbf{Row/Column Tracking:} Maintains boolean arrays to track selected rows and columns, preventing redundant selections.
%     \item \textbf{Adaptive Pivot Selection:} Selects pivot columns based on residual magnitude and pivot rows based on column residual magnitude.
%     \item \textbf{Robust Termination:} Multiple stopping criteria including tolerance check, pivot threshold, and maximum rank.
%     \item \textbf{Correct Quaternion Arithmetic:} Properly handles quaternion operations including multiplication with  transpose (`.''`).
% \end{enumerate}

\end{document}